\documentclass[12pt]{article}
\usepackage[utf8]{inputenc}
\usepackage{amsmath, dsfont, amssymb, bm, amsthm}
\usepackage{graphicx, epstopdf, geometry,color, subfigure}
\usepackage{float}

\def\R{\mathbb R}
\def\Z{\mathbb Z}

\newcommand{\ds}{\displaystyle}

\newcommand{\Ll}{\mathcal{L}}

\newcommand{\p}{\partial}

\newcommand{\tpsi}{\tilde{\psi}}

\newtheorem{theorem}{Theorem}[section]
\newtheorem{lemma}[theorem]{Lemma}

\newtheorem{proposition}[theorem]{Proposition}

\newtheorem{remark}[theorem]{Remark}

\definecolor{aquamarine}{rgb}{0.13, 0.68, 0.8}

\usepackage{enumitem}
\usepackage[english]{babel}
\usepackage[babel=true]{csquotes} 
\usepackage{varioref} 
\usepackage{hyperref} 
\hypersetup{
    colorlinks=true,
    linkcolor=blue,
    }
\usepackage{soul}
\setstcolor{blue}

\newcommand{\pth}[1]{\left(#1\right)}
\newcommand{\cro}[1]{\left[#1\right]}
\newcommand{\acc}[1]{\left\{#1\right\}}
\newcommand{\abs}[1]{\left|#1\right|}
\newcommand{\dabs}[1]{\left\|#1\right\|}

\newcommand{\scal}[1]{\left<#1\right>}

\newcommand{\de}{\,\mathrm{d}}
\newcommand{\dr}{\partial}
\newcommand{\esp}{\hspace{1cm}}
\newcommand{\eps}{\varepsilon}

\newcommand{\drift}{P}

\newcommand{\mca}{\mathcal{A}}
\newcommand{\mce}{\mathcal{E}}
\newcommand{\mcl}{\mathcal{L}}
\newcommand{\mcc}{\mathcal{C}}
\newcommand{\mct}{\mathcal{T}}
\newcommand{\mcm}{\mathcal{M}}

\begin{document}

\title{Impact of diffusion mechanisms on persistence and spreading}

\author{N. Boutillon$^{\footnotesize\hbox{a,b}}$, Y.-J. Kim$^{\footnotesize\hbox{c}}$, L. Roques$^{\footnotesize\hbox{a}}$\\
\\
\footnotesize{$^{\hbox{a }}$INRAE, BioSP, 84914, Avignon, France}\\
\footnotesize{$^{\hbox{b }}$Aix Marseille Univ, CNRS, I2M, Marseille, France}\\
\footnotesize{$^{\hbox{c }}$Department of Mathematical Sciences, KAIST,} \\
\footnotesize{291 Daehak-ro, Yuseong-gu, Daejeon,
34141, Korea}
}

\date{}

\maketitle

\begin{abstract}

   We examine a generalized KPP equation with a \enquote{$q$-diffusion}, which is a framework that unifies various standard linear diffusion regimes: Fickian diffusion ($q = 0$), Stratonovich diffusion ($q = 1/2$), Fokker-Planck diffusion ($q = 1$), and nonstandard diffusion regimes for general $q\in\R$. Using both analytical methods and numerical simulations, we explore how the ability of persistence (measured by some principal eigenvalue) and how the asymptotic spreading speed depend on the parameter $q$ and on the phase shift between the growth rate $r(x)$ and the diffusion coefficient $D(x)$.

Our results demonstrate that persistence and spreading properties generally depend on $q$:  for example, appropriate configurations of $r(x)$ and $D(x)$ can be constructed such that $q$-diffusion either enhances or diminishes the ability of persistence and the spreading speed with respect to the traditional Fickian diffusion.
We
show that the spatial arrangement of $r(x)$ with respect to $D(x)$ has markedly different effects depending on whether $q > 0$, $q = 0$, or $q < 0$. The case where $r$ is constant is an exception: persistence becomes independent of $q$, while the spreading speed displays a symmetry around $q = 1/2$. 

This work underscores the importance of carefully selecting diffusion models in ecological and epidemiological contexts, highlighting their potential implications for persistence, spreading, and control strategies. \\



\end{abstract}


\section{Introduction}

We consider the following reaction-diffusion equations, indexed by a parameter $q\in\R$:
\begin{equation} 
    \label{eq:main_KPP}
    \p_t u(t,x) = \p_x \left( D(x)^{1-q} \p_x(D(x)^q u)\right) + f(x,u), \quad t > 0, \ x \in \mathbb{R}.
\end{equation}
Here, the function $x\mapsto D(x)$ is positive and $1$-periodic, the function $x\mapsto \dr_uf(x,0)$ is $1-$periodic, and $f$ satisfies the KPP condition (we will give the precise definition below).
Equivalently, \eqref{eq:main_KPP} can be written:
\begin{equation} 
    \label{eq:main_KPP_bis}
  \p_t u(t,x) = \p_x \left( D(x) \p_x u\right) + q \p_x \left( D'(x)  u \right) + f(x,u), \quad t > 0, \ x \in \mathbb{R}.
\end{equation}
Our primary goal is to analyze the persistence and spreading properties of the solutions $u(t,x)$ for varying values of $q \in \R$.
In Equation~\eqref{eq:main_KPP}, the term $\p_x \left( D(x)^{1-q} \p_x(D(x)^q u)\right)$ can be referred to as the \emph{``$q-$diffusion operator"}, and appears in several recent papers \cite{AlfGil22,HilKan24,KimLee24,NinNak24}.
Let us describe why $q$-diffusion operators have recently raised so much interest.

Since Einstein's seminal work on Brownian motion \cite{Ein05}, the parabolic partial differential equation (PDE)
$\partial_t u(t,x) = D \partial_{xx} u(t,x)$
with constant diffusion coefficient $D$ has been recognized for its solid physical underpinnings:
the solution $u(t,x)$ is typically interpreted as the density of particles undergoing independent Brownian motions.
From this starting point, there are several ways to introduce heterogeneity in the movement of particles.

A first way to introduce heterogeneity in the movement of particles in~\eqref{eq:main_KPP} is to use the Fokker-Planck equation associated with an Itô stochastic differential equation (SDE), formulated as 
\begin{equation}\label{eq:ito_sde}
\de X_t =  \sigma(X_t) \, \de W_t.
\end{equation}
Here, $\de W_t$ denotes a forward increment in time of a standard Brownian motion. 
The SDE~\eqref{eq:ito_sde} modifies Brownian motion by incorporating a spatio-temporal variability in the movements of particles, as indicated by the position-dependent random driving force $\sigma(X_t)$.
For instance, $\sigma$ takes lower values in regions where the motion of particles is slower.
The associated Fokker-Planck equation describes the dynamics of the transition probability density $u(t,x)=P(X(t) = x | X(0) = x_0)$ (see e.g.\ the textbook \cite{Gar09}): 
\begin{align}
\partial_t u(t,x) = \partial_{xx} (D(x) \, u), \quad t > 0, \quad x \in \mathbb{R},
\end{align}
with Dirac initial condition at $x_0$,  where $D(x)=\sigma^2(x)/2$ and $\partial_{xx} (D(x) \, u)$ corresponds to the diffusion term in \eqref{eq:main_KPP} with $q=1$. The $1$-diffusion will be called, throughout this paper, the \emph{Fokker-Planck diffusion}.

A second standard way to introduce heterogeneity in the movement of particles is to use Fickian diffusion~\cite{Fic55}.
Fickian diffusion is commonly derived from the physical interplay between particle flux, concentration, and concentration gradient~\cite{Cus09,Far93}.
Particles naturally migrate from regions of higher density to regions of lower density, so this diffusion eventually leads to a homogeneous density of particles. 
More precisely, Fickian diffusion is derived from Fick's Law, which posits that the particle flux is proportional to the concentration gradient. In this context, the proportionality coefficient $D(x)$ represents the rate at which particles diffuse through a medium, accounting for spatial variations in the medium properties. 
This construction leads to the following PDE:
\begin{align}
\partial_t u(t,x) = \partial_{x} (D(x) \, \p_x u), \quad t > 0, \quad x \in \mathbb{R},
\end{align}
which corresponds to the diffusion term in \eqref{eq:main_KPP} with $q=0$. The $0$-diffusion will be called, throughout this paper, the \emph{Fickian diffusion}.



Third, Wereide \cite{Wer14} formulated a diffusion law corresponding to the case where $q = 1/2$. This law incorporates the phenomenon of \textit{thermo-diffusion}, where particles migrate not only in response to a concentration gradient but also under the influence of a temperature gradient. Considering the SDE \eqref{eq:ito_sde} but with Stratonovich integration instead of Itô integration, the density $u(t,\cdot)$ of the law of the particle at time~$t$ satisfies
\begin{align}
\partial_t u(t,x) = \partial_{x} \pth{\sqrt{D(x)}\partial_x\pth{\sqrt{D(x)} \, u}}, \quad t > 0, \quad x \in \mathbb{R},
\end{align}
where $D(x)=\sigma^2(x)/2$ corresponds to the diffusion term in \eqref{eq:main_KPP} with $q=1/2$. The $1/2$-diffusion will be called the \emph{Stratonovich diffusion}.
More recently, Kim and Seo \cite{KimSeo21} developed heterogeneous diffusion equations from reversible kinetic systems and compared various diffusion laws through a theoretical experiment. Notably, they derived the Stratonovich diffusion law as a particular instance of their model.

\medskip

The Fokker-Planck, Fickian and Stratonovich diffusion laws can also emerge as limits of stochastic processes.
The Itô integral interpretation of the equation $\de X_t =  \sigma(X_t) \, \de W_t$ assumes that the jump $X_{t+\delta t}-X_t$ 
(for a time step $\delta t$ going to $0$)
depends on the intensity $\sigma(X_t)$ of the stochastic process at the departure point $X_t$.  
This interpretation leads to the Fokker-Planck diffusion. 
The analysis in~\cite{Vis08} indicates that adopting an alternative assumption, where the length of the jump depends on the value of~$\sigma$ at an intermediate point $q X_t + (1-q) X_{t+\delta t}$ (for $q \in \R$), leads to the diffusion term $\partial_x \left( D(x)^{1-q} \partial_x(D(x)^q u)\right)$. 
More precisely, $X_t$ being defined, $X_{t+\delta t}$ is implicitly defined, for $\delta t$ small enough, by
\[X_{t+\delta t}-X_t=\sigma\pth{q X_t + (1-q) X_{t+\delta t}}(W_{t+\delta t}-W_t).\]
Then, Fickian diffusion ($q=0$) is characterized by a dependency on the arrival point and Stratonovich diffusion ($q=1/2$) is characterized by a dependency on the midpoint.
 Alternative constructions of $q$-diffusions have been explored in~\cite{Oku86}, relying upon jump rates instead of jump lengths, and more recently in~\cite{AlfGil22}, where a combination of variable jump rates and jump lengths has been investigated.
 In~\cite{PotSch14}, the authors focused particularly on the cases $q=1$ (Fokker-Planck diffusion) and $q=-1$. The latter is interpreted as \emph{attractive dispersal}, i.e. the probability of jumping to one's neighbor depends on the environment at that neighbor.
We emphasize that the derivations in these works rely on different factors of random walk systems (length of jumps, waiting time, probability of jumping). The reference point at which this factor is considered (middle point, end point, etc.) depends not only on~$q$ but also on the original probabilistic model.

\paragraph{Problematic of this work.}
In the field of reaction-diffusion equations, both Fokker-Planck and Fickian diffusion operators are commonly encountered to model heterogeneous diffusions (for instance \cite{BenFeu14,CanCos03,KinKaw06,Shikaw97} use Fickian diffusion and \cite{Okulev02,Pat53} use Fokker-Planck diffusion). There appears to be a preference for the use of Fickian diffusion $\partial_x(D(x)\partial_x u)$, due to its self-adjoint nature, which facilitates mathematical computations.

The above presentation indicates that $\partial_{xx}\!\big(D(x)u\big)$ the Fokker-Planck  operator 
 should be more relevant for the description of Brownian-like individual movements when individuals move
differently based on the local information that is accessible to them. 
Several authors advocate for the use of the Fokker-Planck diffusion in biological systems \cite{HanEkb01,PotSch14,RoqAug08,Tur98,Vis08}.
In contrast, the Fickian diffusion operator is thought to be more suitable for passive physical phenomena, such as the dilution of a dye in a liquid or the diffusion of heat in a heterogeneous medium. 
As scientific authors in mathematical ecology often seem to use them interchangeably, depending on the technical constraints related to mathematical proofs or based on the modeling works to which they refer, it seems important to understand the concrete effect related to this choice. 
For example, as underscored by \cite{BenMal16} (refer also to \cite{Roq13}), it is well-established that the formation of patterns in spatially inhomogeneous diffusive systems is significantly more pronounced when diffusion follows the Fokker-Planck law, as opposed to Fick's law.

Reaction-diffusion models predominantly aim at examining the persistence or extinction of populations and their spatial spreading dynamics. In this context, the main goal of our work is to analyze the persistence and spreading dynamics of the model in Equation~\eqref{eq:main_KPP}, with a focus on the parameter $q$. Specifically, we investigate whether the assumption of a Fickian or Fokker-Planck diffusion significantly affects the outcomes.

\paragraph{Layout.} 
In Section~\ref{s:standard}, we give our general assumptions and standard properties of equations of the form~\eqref{eq:main_KPP} under a KPP assumption.
Next, in Section~\ref{s:results}, we give our main results, along with the main ideas of proofs.
We add an analysis of some numerical results.
We discuss our results in Section~\ref{s:discussion}.
Last, Section~\ref{s:proofs} is devoted to the proofs.

\section{Main assumptions and standard properties}\label{s:standard}

\subsection{Assumptions}

The diffusion coefficient $D(x)$ is assumed to be positive, 1-periodic and of class $\mcc^{2,\alpha}$ (with $\alpha > 0$).
The function $f: \mathbb{R} \times \mathbb{R}_+ \to \mathbb{R}$ is assumed to be of class $\mcc^{1,\alpha}$ with respect to $(x,u)$ and $\mcc^2$ with respect to $u$. We also assume that $f$ is 1-periodic with respect to~$x$ and that $f(x,0) = 0$ for all $x \in \mathbb{R}$. We define
\begin{equation*}
    r(x) := \left. \frac{\partial}{\partial u} f(x, u) \right|_{u=0}.
\end{equation*}
When $u$ represents a population density, the coefficient $r(x)$ is often interpreted as the intrinsic growth rate of the population.

Furthermore, throughout the paper, we assume the following, which is a (slightly stronger) variant of the KPP assumption:
\begin{equation}\label{f1}
\forall x \in \mathbb{R},\ u \mapsto \frac{f(x,u)}{u} \text{ is decreasing in } u > 0
\end{equation}
and, for all $q\in\R$,
\begin{equation}\label{f2}
\exists M_q \ge 0,\ \forall u \ge M_q,\ \forall x \in \mathbb{R},\ f(x,u) + q D''(x) u \le 0.
\end{equation}
These assumptions, along with the instability assumption~\eqref{eq:persist_cond} below, ensure the existence and uniqueness of a positive stationary state to~\eqref{eq:main_KPP}, i.e. a periodic solution $p>0$ to
\[\dr_x(D^{1-q}\dr_x(D^qp))+f(x,p)=0, \qquad x\in\R,\]
see~\cite{BHR07}.
Examples of functions $f$ satisfying (\ref{f1}-\ref{f2}) for all $q\in\R$ include $f(x,u) = r \, u (1-u)$ or $f(x,u) = u(r(x)-\gamma(x)u)$, where $r$ and $\gamma$ are $\mcc^{1,\alpha}$ periodic functions.
The regularity assumptions on $r$ and $\gamma$ (and thus on $f$) could be relaxed, for example, to $L^\infty$, but it would require us to deal with Sobolev spaces $W^{2,p}$ rather than Hölder spaces $\mcc^{2,\alpha}$.

\subsection{Standard persistence and spreading properties}\label{ss:standard_prop}

For $q\in\R$, we let $\mcl_q^0$ be the operator defined by
\[\mcl_q^0\psi:= \dr_x(D^{1-q}\dr_x(D^q\psi))+r\psi.\]
For $\lambda\in\R$, we define the operator $\mathcal{L}_{q}^{\lambda}$ as follows:
\begin{equation*}
  \mcl_q^{\lambda}\psi(x):=e^{\lambda x}\mcl_q^0\pth{y\mapsto e^{-\lambda y}\psi(y)}.
\end{equation*}
More explicitly, we have:
\begin{equation*}
    \mathcal{L}_q^\lambda\, : \, \psi \mapsto D \psi'' + \left( (1+q) D' - 2 \lambda D \right) \psi' + \left( q D'' + \lambda^2 D - (1+q) \lambda D' + r \right)\psi.
\end{equation*}
We define the principal eigenvalue $k^{\lambda}_q$ of $\mcl_q^{\lambda}$ as the unique real number for which there exists a $1-$periodic function $\varphi > 0$ satisfying
\[\mathcal{L}_q^\lambda \varphi = k_q^{\lambda}\varphi.\]
For each $q,\lambda\in\R$, the existence and uniqueness of the pair $(\varphi,k_q^{\lambda})$ (up to multiplication of $\varphi$ by a constant) is guaranteed by the Krein-Rutman theory.
For the sake of explicitness, we will usually use the notations $\mcl_q^{\lambda}[ r; D ]$ and $k_q^{\lambda}[r;D]$ to insist on the dependency on the coefficients $r$ and $D$.

We now state two important properties of the principal eigenvalues $k_q^{\lambda}[r;D]$ ($\lambda\in\R$) to justify why their properties are crucial to the study of the solutions of Equation \eqref{eq:main_KPP}.


First, we focus on the {persistence} of the population.
By \enquote{persistence}, we will refer to the property that starting from any bounded and continuous function~$u_0$ satisfying $u_0 \geq 0$ and $u_0 \not\equiv 0$, the solution $u(t,x)$ to the Cauchy problem \eqref{eq:main_KPP} with initial condition $u(0, \cdot) = u_0$ satisfies
\[\liminf_{t\to+\infty}\sup_{x\in\R}u(t,x)>0.\]
Persistence holds if, and only if
\begin{equation} \label{eq:persist_cond}
        k^{0}_q[r;D] > 0,
\end{equation}
see~\cite{BerHamRoq05a} for $q=0$, and see~\cite[Appendix A]{BouHamRoq25} for general~$q$.
Under condition~\eqref{eq:persist_cond}, assumptions~\eqref{f1}-\eqref{f2} ensure that the solution converges in large time to the unique positive stationary solution $p$. 
On the other hand, if $k^0_q[r;D]\leq 0$, then, starting from any bounded continuous and nonnegative initial condition $u_0$ with compact support, the solution converges uniformly to $0$.
This means, from the point of view of modeling, that the population gets extinct.

In this article, we call $k_q^0[r;D]$ the \enquote{ability of persistence} for the parameters $q$, $r$ and~$D$.
This terminology comes from the fact that a population described by Equation~\eqref{eq:main_KPP} is more resilient to constant perturbations of its growth rate when $k_q^0[r;D]$ is higher. Specifically, if $\kappa>0$ is a constant, then $k_q^0[r-\kappa;D]=k_q^0[r;D]-\kappa$ becomes negative if  $k_q^0[r;D]<\kappa$.  
Consequently, a higher value of $k_q^0[r;D]$ increases the likelihood of persistence under a constant random perturbation of~$r$.

Second, let us define the spreading speed (to the right) $c^*_q[r;D]$ by:
\begin{equation}\label{eq:garfre}
        c^*_q[r;D] = \inf_{\lambda > 0} \frac{k_q^{\lambda}[r;D]}{\lambda}>0.
\end{equation}
Equation~\eqref{eq:garfre} is called the \emph{Freidlin-Gärtner formula} for the spreading speed.
If  \allowbreak $c^*_q[r;D] \!>\! 0$, we say that the population spreads to the right. If $u_0 \not \equiv 0$  is continuous, nonnegative and has a compact support, then for each $ w \in(0, c^*_q[r;D])$,
\[
\liminf_{t \to +\infty} u(t, wt) > 0,
\]
and for each $ w > c^*_q[r;D]$ ,
\[
\lim_{t \to +\infty} u(t, wt) = 0.
\]
This means that if an observer moves to the right at a speed larger than $ c^*_q[r;D]$, they will see the solution decay to $0$, while if the observer travels at a speed smaller than $ c^*_q[r;D]$, they will observe a density that remains bounded from below by a positive constant.
More precisely, the solution $u$ converges locally uniformly to the unique stationary state $p$ of  Equation~\eqref{eq:main_KPP}.

We know from \cite{BerHamRoq05b,Nad09a} that $\lambda \mapsto k^{\lambda}_0[r;D]$ is strictly convex and reaches its minimum at $\lambda=0.$ Using Proposition~\ref{prop:compar_k} below, we will see that $\lambda \mapsto k^{\lambda}_q[r;D]$ has the same properties for all $q\in\R$.
Thus, if the condition for persistence is satisfied, the condition  for spreading is also satisfied. Therefore, we have the following well-known, yet significant, equivalence:
\begin{center}
  \emph{The condition for persistence $k^{0}_q[r;D] > 0$ is equivalent to \\
  the condition for spreading  $c_q^*[r;D]>0$.}
\end{center}

In our context, the existence of a spreading speed and the formula \eqref{eq:garfre} follow from~\cite{Berhamnad08,BerHamNad05d, Wei02}.
This framework was originally proposed in the work by Gärtner and Freidlin~\cite{GarFre79} when $q=0$.

Finally, let us mention that the spreading speed to the left can be defined as well, and the same type of Freidlin-Gärtner formula holds \cite{Berhamnad08,BerHamNad05d}:
\begin{equation}\label{eq:garfre_left}
c^{*,\,{left}}_{q}[r;D]= \inf_{\lambda>0}\frac{k_q^{-\lambda}[r;D]}{\lambda}.  
\end{equation}
In our context, the spreading speed to the left and the spreading speed to the right coincide (see Proposition~\ref{prop:left_right}). This is due to the fact that the drift term in~\eqref{eq:main_KPP_bis} has zero integral on a period.

\section{Results}\label{s:results}

We now state our results.
We first prove that in general, the choice of the parameter~$q$ of the $q$-diffusion operator has an influence on the ability of persistence and the spreading speed.
Second, we consider a constant reaction term~$r$; then the ability of persistence does not depend on~$q$.
Regarding the spreading speed, we prove that there is a symmetry around the Stratonovich diffusion. In particular, the Fickian diffusion and the Fokker-Planck diffusion are symmetric with respect to the Stratonovich diffusion, and thus give the same spreading speed.
Third, we consider general~$r$ again and we give monotonicity and non-monotonicity properties of the spreading speed with respect to the parameter~$q$.
Last, we provide observations from numerical simulations.

Throughout this section, we work under the assumption that $r\in\mcc^{0,\alpha}(\R)$ is $1-$\allowbreak periodic and $D\in\mcc^{2,\alpha}(\R)$ is $1-$periodic and positive.

\subsection{Different behaviors with respect to \texorpdfstring{$q$}{q}}

In our first result, we show that the choice of the value of $q$ may indeed have an impact on the ability of persistence and on the spreading speed.
In particular, through appropriate choices of coefficients $r$ and $D$, the $q$-diffusion operator can give rise to a higher or lower ability of persistence and spreading speed with respect to Fickian diffusion ($q = 0$).

\begin{theorem}\label{thm:choose_r_d}
  Consider a nonconstant $D$ and let $q\in \R \setminus \{0\}.$
  \begin{enumerate}
      \item[(i)] There exists $r$ for which $k_q^0[r;D] < k_0^0[r;D]$ and $c_q^*[r;D]<c_0^*[r;D]$;
      \item[(ii)] There exists $r$ for which $k_q^0[r;D] > k_0^0[r;D]$ and $c_q^*[r;D]>c_0^*[r;D]$.
  \end{enumerate}
  Consider a nonconstant $r$ and let $q\in \R \setminus \{0\}.$
  \begin{enumerate}
      \item[(iii)] There exists $D$ for which $k_q^0[r;D] < k_0^0[r;D]$; 
        \item[(iv)] There exists $D$ for which $k_q^0[r;D] > k_0^0[r;D]$. 
  \end{enumerate}
\end{theorem}

\begin{remark}
Using the same methodology as in item~$(i)$ of the theorem, it can be shown that for every non-constant $D$ and for $\abs{q} > \abs{\tilde{q}}$, there exists a function $r$ such that
\[
k_{q}[r;D] < k_{\tilde{q}}[r;D], \quad c_q^*[r;D] < c_{\tilde{q}}^*[r;D].
\]
Our proof does not allow us to obtain other comparisons with general values of~$q$ and~$\widetilde{q}$. 
\end{remark}

Let us now describe the main tools that we will use to prove Theorem~\ref{thm:choose_r_d}.
The first ingredient for items \emph{(i)} and \emph{(ii)} is the following proposition, which gives a relationship between the eigenvalues $k_q^{\lambda}[r;D]$ for different values of $q$.

\begin{proposition} \label{prop:compar_k}
For all $q\in \R$ and $\lambda \in \R,$ we have 
\[k_q^\lambda[r;D] = k_0^\lambda[ r - h_q ; D ],  \]
with $\ds h_q:=-\frac{q}{2}D'' +\frac{q^2}{4} \frac{(D')^2}{D}$.
\end{proposition}

This proposition shows that, in terms of persistence and spreading, Equation~\eqref{eq:main_KPP}  behaves as a Fisher-KPP equation with Fickian diffusion, but where the growth term~$r$ is replaced with $r- h_q$.
Due to the periodicity of $D$, the arithmetic mean values of~$r$ and $r-h_q$ can be compared: $\langle r- h_q\rangle_A \le \langle r\rangle_A $, and the inequality is strict if $D$ is not constant. Additionally, $\langle r- h_q\rangle_A$ is a decreasing function of $q$. On the other hand, since $D$ is periodic, there exists $x_0$ in $[0,1]$ such that $-h_q(x_0)>0$ and therefore $r(x_0)<r-h_q(x_0)$. With such considerations, we will be able to conclude the first part of items \emph{(i)} and \emph{(ii)} on persistence. 

The second ingredient for items \emph{(i)} and \emph{(ii)}  is the following proposition, which allows us (in situations of persistence, but close to extinction) to extend a comparison on $k_q^0[r;D]$ to a comparison on $c_q^*[r;D]$.

\begin{proposition}\label{prop:compar_persist_speed}
     Let $q,\,\tilde{q}\in \R$ and $r,$ $D$ be fixed. Assume that $k_q^0[r;D] > k_{\tilde{q}}^0[r;D]>0$. 
     Then, there exists $\kappa_0 \in (0,k_{\tilde{q}}^0[r;D])$ such that, for all $\kappa \in (\kappa_0, k_{\tilde{q}}^0[r;D]),$ $$c_q^*[ r - \kappa; D ] > c_{\tilde{q}}^*[ r - \kappa; D ] >0.$$
\end{proposition}

Finally, for items \emph{(iii)} and \emph{(iv)}, which only deal with a persistence property, we use the following proposition.

\begin{proposition} \label{prop:largeD}
    Let $D$ and $r$ be fixed. For all $q\in \R$, we have:
    \[\lim_{B\to \infty}k^{0}_q[ r; B\, D ]= \int_0^1 r \, \frac{D^{-q}}{\int_0^1 D^{-q}}.\]
\end{proposition}

On the one hand, Proposition~\ref{prop:largeD} implies that provided $r>0$ and $B$ is sufficiently large, diffusion functions of the form $B\, D=B \, r^{1/q}$  lead to a higher ability of persistence for the Fickian diffusion than for other values of $q$, due to the arithmetic-harmonic mean inequality. The function $Br^{1/q}$ is \enquote{in phase} with~$r$ for $q>0$ and is \enquote{out of phase} with~$r$ for $q<0$. We will prove~$(iii)$ in this way.

To prove the reverse inequality $(iv)$, we will consider functions of the form $D_1 := B\, D$, with $B$ large enough, and for a function $D$ such that $1/D$ is concentrated around the maximum of $r$ if $q>0$, or $D$ is concentrated around the maximum of $r$ if $q<0$. Again, $(iv)$ will follow from Proposition~\ref{prop:largeD}.

Since $r$ is fixed in~$(iii)$ and~$(iv)$, it is not possible in general to choose $D$ such that the population is close to extinction, and to use Proposition~\ref{prop:compar_persist_speed} to extend the result on persistence to a result on the spreading speeds, as we did for items~$(i)$ and~$(ii)$.
However, for every non-constant $r$, it can be stated that there exist pairs $(D,\kappa)$ 
(where~$\kappa$ is a constant) such that $c_q^*[r-\kappa;D] < c_0^*[r-\kappa;D]$, and pairs $(D,\kappa)$ for which $c_q^*[r-\kappa;D] > c_0^*[r-\kappa;D]$.

\subsection{Constant growth rate \texorpdfstring{$r$}{}}

In this section, we focus on the special case where the growth rate $r$ is constant.
The main result of this section is the following theorem,
which shows that the spreading speeds for Fickian diffusion ($q=0$) and Fokker-Planck diffusion ($q=1$) coincide, as they are, in some sense, symmetric. The center of this symmetry is the Stratonovich diffusion ($q=1/2$).

\begin{theorem}\label{thm:r_constant}
  Assume that $r>0$ is constant. We have the following properties:
  \begin{enumerate}
  \item[(i)] The function $q\mapsto c^*_{q}[r;D]$ is symmetric around $1/2$, that is, for all $\tilde{q}\in\R$,
  \[c^*_{1/2-\tilde{q}}[r;D]=c^*_{1/2+\tilde{q}}[r;D].\]
  In particular, $c^*_0[r;D]=c^*_1[r;D]$.
\item[(ii)] For $q=1/2$, we have the explicit formula:
  \[c^*_{1/2}[r;D]=2\sqrt{r}\times\scal{ \sqrt{D}}_H,\esp\text{with }\scal{\sqrt{D}}_H := \pth{\int_{0}^1\frac{1}{\sqrt{D}}}^{-1}.\]
\item[(iii)]   The maximal speed is reached at $q=1/2$:
  \[c^*_{1/2}[r;D]=\max_{q\in\R}c^*_q[r;D].\]
  \end{enumerate}
\end{theorem}


Item~$(i)$ of this theorem is somewhat unexpected, and the ideas of its proof are explained below.
Item $(ii)$ is also remarkable, because it allows us to establish a connection to previous work on slowly varying media, which we now detail.

For $L>0$, when $D(x)$ is replaced with $D(x/L)$, we can define the operators $\mcl_q^{\lambda}(L)$ analogously to $\mcl_q^{\lambda}$; 
the operator $\mcl_q^{\lambda}(L)$ has a principal eigenvalue $k_q^{\lambda}(L)$, depending on~$L$,
and associated with a $L-$periodic principal eigenfunction.
  The criterion for persistence still holds, and the spreading speed $c_q^*(L)$ is given by the Freidlin-Gärtner formula~\eqref{eq:garfre}, i.e.
  \[c_q^*(L)=\inf_{\lambda>0}\frac{k_q^{\lambda}(L)}{\lambda}.\]
  Item~$(ii)$ of Theorem~\ref{thm:r_constant} gives an explicit formula for $c_{1/2}^*(1)$;
  this explicit formula also appears in the context of slowly varying environments.
  Namely, as a consequence of~\cite[Theorem 2.3]{HamNadRoq11}, we have
  \[\lim_{L \to +\infty} c^*_0(L)=\inf_{\lambda>0}\pth{\frac{r}{\lambda}+ \lambda\scal{\sqrt{D}}^2_H}
  =2\sqrt{r}\times\scal{ \sqrt{D}}_H,\]
  which is equal to $c^*_{1/2}(1)$.
  Moreover,
  in the proof of Theorem~\ref{thm:r_constant} $(ii)$, we will also prove that $c^*_{1/2}(L)$ is in fact independent of~$L$: thereby, for $q=1/2$ and constant~$r$, the \emph{spreading speed does not depend on the period}.
Let us close this discussion about different periods and keep, from now on, $L=1$.

%

The proof of item~$(ii)$ of Theorem~\ref{thm:r_constant} relies on three key propositions. 
The first proposition turns the heterogeneous diffusion into a homogeneous diffusion with a drift term depending on $q$.
This transformation uses the space variable change $y = h(x)$, where
\[h(x):=\int_0^x\frac{\de x'}{\sqrt{D(x')}}.\]
Note that $1/h(1)=\scal{\sqrt{D}}_H$. 

\begin{proposition}\label{ppn:equiv_unif_laplacian}
Let $R(y):=r\pth{h^{-1}(y)}$, $\drift(y):=\ln \pth{D\pth{h^{-1}(y)}}$ and $s_q:=\frac{1}{2}-q$. Let~$\widehat{k}_q^{\lambda}$ be the $h(1)-$periodic principal eigenvalue of the operator $\widehat{\mcl}_q^{\lambda}$ defined by
\[\widehat{\mcl}_q^{\lambda}:\Phi\mapsto\dr_{yy}\Phi - \dr_y\cro{\pth{2\lambda+s_q\drift'}\Phi}+\left(R+\lambda s_q\drift'+\lambda^2\right)\Phi.\]
Then for all $\lambda\in\R$ and $q\in\R$, 
\[k_q^{\lambda}[r;D]=\widehat{k}_q^{\lambda\scal{\sqrt{D}}_H}.\]
\end{proposition}

As a consequence of Proposition~\ref{ppn:equiv_unif_laplacian}, we have $k_q^0 = \widehat{k}_q^0$ and, using~\eqref{eq:garfre}:
\[c^*_q[r;D]=\scal{\sqrt{D}}_H\times\inf_{\lambda>0}\frac{\widehat{k}_q^{\lambda}}{\lambda}.\]
This result highlights a symmetry between Fickian diffusion ($q = 0$) and Fokker-Planck diffusion ($q = 1$), since $s_0 = -s_1$. However, demonstrating that
$\widehat{k}_{0}^{\lambda} = \widehat{k}_{1}^{\lambda}$ or, more generally, that $\widehat{k}_{1/2 - \tilde{q}}^{\lambda} = \widehat{k}_{1/2 + \tilde{q}}^{\lambda}$ for $\tilde{q} \in \mathbb{R}$ remains a challenging task.
The crucial ingredient is the following result from \cite{BouHamRoq25}, which implies that, given a periodic Schr\"odinger operator with an advection term and a potential term proportional to the advection term, the principal eigenvalue remains invariant under a sign change of the advection term.
\begin{proposition}[\cite{BouHamRoq25}, Theorem 2.2]\label{prop:VP}
Let $a\in \mathcal{C}^{0,\alpha}(\R) $ and $b\in \mathcal{C}^{1,\alpha}(\R)$ be two $1-$periodic functions. Consider the
two operators acting on $1-$periodic functions:
 \begin{align*}
     \mcm^- : \Phi\mapsto\dr_{yy}\Phi - \dr_y(b\Phi)+a \Phi \hbox{ and }
     \mcm^+ : \Phi\mapsto\dr_{yy}\Phi + \dr_y(b\Phi)+a \Phi.  
 \end{align*}
If there exist $\beta\in\R$ and $\gamma\in\R$ such that $a=\beta+\gamma b$, then $\mcm^-$ and $\mcm^+$ have the same principal eigenvalue.
 \end{proposition}

In fact, Theorem 2.2 in \cite{BouHamRoq25} deals with the adjoint operators of~$\mcm^-$ and~$\mcm^+$, but the result holds also for~$\mcm^-$ and~$\mcm^+$, since the principal eigenvalue of an operator coincides with the principal eigenvalue of its adjoint.

Last, we will need to use the equality between the spreading speed to the left and the spreading speed to the right.
Owing to the nonsymmetric nature of Equation~\eqref{eq:main_KPP} with respect to the variable $x$, this equality is not obvious. We state it in the following proposition. 
\begin{proposition}\label{prop:left_right}
  For all $q\in \R$ and $\lambda \in \R,$ $k^{-\lambda}_q[r;D] = k^{\lambda}_q[r;D]$. 
  As a consequence, the spreading speed to the left and the spreading speed to the right coincide:
\[c^{*}_q[r;D]=\inf_{\lambda>0}\frac{k_q^{\lambda}[r;D]}{\lambda}=\inf_{\lambda>0}\frac{k_q^{-\lambda}[r;D]}{\lambda}=c_q^{*,\,left}[r;D].\]
\end{proposition}

\subsection{(Non-)monotonicity in \texorpdfstring{$q$}{q} and limiting behavior as \texorpdfstring{$q\to\pm\infty$}{}}

We now go back to the case of general $1-$periodic $r\in\mcc^{0,\alpha}(\R)$ and focus on the dependency of $k_q^0[r;D]$ and $c^*_q[r;D]$ on $q$.
Again, Proposition~\ref{ppn:equiv_unif_laplacian} will be central in the proofs. 

\begin{theorem}\label{thm:limits}
  \begin{enumerate}[label=$(\roman*)$]
  \item Assume that all extremal points of $D$ are non-degenerate (i.e. $D''\neq0$ at these points). Let $\underline{\mca}$ be the set of points at which $D$ reaches a local minimum and let $\overline{\mca}$ be the set of points at which $D$ reaches a local maximum. Then,
  as ${q}\to\pm\infty$, the ability of persistence $k_q^0$ converges to the following limits:
    \[\lim_{q\to+\infty}k_q^0[r;D]=\max_{x\in\underline{\mca}}r(x),
    \esp
    \lim_{q\to-\infty}k_q^0[r;D]=\max_{x\in\overline{\mca}}r(x);\] 
  \item Assume that $D$ is not constant. Then, as ${q}\to\pm\infty$, the spreading speed $c^*_q[r;D]$ converges to $0$:
\[\lim_{q\to\pm\infty}c^*_q[r;D]=0.\]  
  \end{enumerate}
  
\end{theorem}

Let us comment on this theorem. The limit as $q\to\pm\infty$ of $k_q^0[r;D]$ is similar to \cite[Theorem 1]{CheLou08}, which dealt with an eigenvalue problem with Neumann boundary condition (instead of periodic boundary conditions). Our proof share similarities with their proof, but has been adapted to the periodic setting.
The second item seems quite surprising, especially when it is compared to the first item: namely, we can have
\[\ds\liminf_{q\to\pm\infty}k_q^0[r;D]>0\]
but
\[\lim_{q\to\pm\infty}c_q^*[r;D]=0.\]
The second item highlights a fundamental difference between homogeneous and heterogeneous diffusion, as for constant $D$, the spreading speed $c_q^*[r;D]$ is independent of~$q$.

The first ingredient of the second item is the \emph{Feynman-Kac formula}.
Let $\overline{u}$ be the solution to the \emph{linear adjoint} Cauchy problem
\begin{equation*}
\p_t \overline{u}(t,x) = D(x)\,\p_{xx} \overline{u} +b(x)\, \dr_x\overline{u}+ r(x)\overline{u}, \quad t > 0, \ x \in \mathbb{R},  
\end{equation*}
with $\overline{u}(0,\cdot)=u_0$ (where $b\in\mcc^{1,\alpha}(\R)$ is also $1-$periodic).
The Feynman-Kac formula implies that $\overline{u}$ can be expressed in terms of a stochastic process $(X_t)_{t\geq0}$ as: 
\begin{equation}\label{eq:feynman_kac_formula}
  \overline{u}(t,x)=\mathbb{E}_x\cro{u_0(X_t)\exp\pth{\int_0^tr(X_{t-s})\de s}}.
\end{equation}
The process $(X_t)_{t\geq0}$ under the expectation $\mathbb{E}_x$ is defined by:
\[\de X_t=b(X_t)\de t+\sqrt{2D(X_t)}\de W_t,\esp X_0=x,\]
where $(W_t)_{t\geq0}$ is a standard Brownian motion.

The second ingredient for the proof of the second item of Theorem~\ref{thm:limits} is the following proposition, which deals with the exit time of an interval by the process~$(X_t)_{t\geq0}$ when the intensity of the drift~$b$ goes to infinity.
We state the result for general Lipschitz drift term~$b$. This regularity assumption ensures the well-posedness of the equation.

\begin{proposition}\label{ppn:exit_time}
  Let $b\in\mcc^{0,1}(\R)$ be a Lipschitz function such that there exists $x_0\in[0,1]$ with $b(x_0)>0$.
  For $s\in\R$, let $(X^s_t)_{t\geq 0}$ be the unique solution of
  \[\de X^s_t=sb(X^s_t)\de t+\sqrt{2}\de W_t, \qquad X_0=1.\]
  Let $\tau^s:={\inf\acc{t\geq 0\ /\ X^s_t\leq 0}}$ be the exit time from $(0,+\infty)$ of $(X^s_t)_{t\geq0}$.
  Then, for all $a>0$, as $s\to+\infty$,
  \[\mathbb{P}\pth{\tau^s\leq a}\to 0.\]
\end{proposition}

\medskip

Let us now state our last main result.
The simulations in Section \ref{sec:num} suggest that $q\mapsto k_q^0[r;D]$ is monotonic.
In fact, the monotonicity of $q\mapsto k_q^0[r;D]$ depends on the situation and does not hold in general. Finally, the monotonicity of  $q\mapsto c^*_q[r;D]$  never holds if spreading can occur  and~$D$ is nonconstant. 

\begin{theorem}\label{thm:monotonicity}
  Let $\kappa \,:\, q\mapsto k_q^0[r;D]$. We have:
  \begin{enumerate}[label=(\roman*)]
  \item If $r$ is a constant, then $\kappa(q)=r$ for all $q\in\R$;
  \item 
    For this item, assume the following: 
    \begin{enumerate}
    \item $D$ and $r$ are monotonic on $[0,1/2]$, $1-$periodic and even (thus $D$ and $r$ are also monotonic on $[1/2,1]$);
    \item $D\in\mcc^{3}(\R)$ and $r\in\mcc^4(\R)$; 
    \item $r'\neq 0$ on $(0,1/2)$; $r''(0)\neq 0$ and $r''(1/2)\neq 0$. 
    \end{enumerate}
    Then $\kappa$ is monotonic. If $D$ and $r$ have the same monotonicity on $[0,1/2]$, then $\kappa$ is nonincreasing. If $D$ and $r$ have opposite monotonicity, then $\kappa$ is nondecreasing;
  \item On the other hand, there exist $r$ and $D$ such that $\kappa$ is not monotonic  (and nonconstant);
  \item Assume that  $D$ is not constant and that  there exists $q\in\R$ such that $k_q^0>0$. Then the function $q\mapsto c^*_q[r;D]$ is not monotonic.
  \end{enumerate}
\end{theorem}
Item $(i)$ of this theorem is standard. Item $(ii)$ shows, similarly to Proposition~\ref{prop:largeD} but in greater details, that increasing $q$ (for instance by shifting from Fickian to Fokker-Planck diffusion) enhances the ability of persistence when $r$  and $D$ are in phase, and decreases it when $r$ and $D$ are out of phase. It is proved by using the Feynman-Kac formula.
The third item is proved by constructing an explicit counterexample.
The fourth item is proved by using Theorem~\ref{thm:limits}.

\subsection{Numerical simulations \label{sec:num}}

We use standard matrix methods in Python (SciPy and ARPACK libraries using the eigs function) together with a finite difference approximation to determine the principal eigenvalue $k_q^\lambda[r;D]$ for discrete values of $\lambda$. This information is then used in conjunction with the Freidlin-Gärtner formula~\eqref{eq:garfre} to yield an approximate value of  $c_q^*[r;D]$. A Jupyter notebook is available at  \href{https://doi.org/10.17605/OSF.IO/GDQVP}{DOI: 10.17605/OSF.IO/GDQVP}. 
This method is expected to be more reliable than the direct simulation of the solution to \eqref{eq:main_KPP} conducted in previous studies (\cite{KinKaw06} for $q=0$ and \cite{BenMal16} for $q\in [0,1]$) for multiple reasons: $(i)$ it does not require a finite domain approximation since the eigenvalue problem is solved over one period; $(ii)$ it avoids reliance on a finite time approximation to estimate the spreading speed; $(iii)$ it circumvents the challenges posed by nonlinearity (the eigenvalue problem defining $k_q^{\lambda}[r;D]$ is linear).

\paragraph{Constant $r$ case.} In the computations presented in Fig.~\ref{fig:speed_vs_q}, we begin by examining the dependency of $c_q^*[r;D]$ on $q$, where $r\equiv r_0$ is a constant. As expected from the result of Theorem~\ref{thm:r_constant},  we find that the speed $q\mapsto c_q^*[r;D]$ is symmetric with respect to $q=1/2$, and is maximal at $q=1/2$. 
This is also consistent with previous results in~\cite{BenMal16} for $q\in [0,1]$, where $q\mapsto c_q^*[r;D]$ is close to its maximum and therefore exhibits a parabolic shape. Outside of this range, we observe that $c_q^*[r,D]$ decays, and converges to $0$ (we get $c_q^*[r,D]\approx 0.01$ with $q=50$) as $q\to \pm \infty$, which is consistent with the result of Theorem~\ref{thm:limits}.

\paragraph{Effect of a phase shift between $r$ and $D$.} Next, we explore how $k_q^0[r;D]$ (the principal eigenvalue determining the ability of persistence)  and the spreading speed $c^*_q[r;D]$ vary with a phase shift $\omega$ between $r$ and $D$, see Fig.~\ref{fig:speed_vs_omega}. More precisely, we set $D(x)=C + r(x+\omega)$, for some constant $C$ ensuring $D>0$, with $\omega \in [0,1]$. 
Here, $\omega=0$ means that $r$ and $D$ are \enquote{in phase} and $\omega=1/2$ means that $r$ and $D$ are \enquote{out of phase}.  
We consider six values of $q$: $q=0$ (Fick), $q=1/2$ (Stratonovich),  $q=1$ (Fokker-Planck), $q=2$, $q=-1$ (called \textit{attractive dispersal} or AD in \cite{PotSch14}), and $q=-1/2$. As expected from Proposition~\ref{prop:largeD} and Theorem~\ref{thm:limits} $(i)$, the effect of the phase shift $\omega$ on $k_q^0[r;D]$ (and therefore on $c^*_q[r;D]$) depends strongly on $q$.

For $q\ge 0$, consistently with previous findings (\cite{KinKaw06} in the Fickian case, see also Theorem~5.1 in \cite{Nad10}), the highest speeds and highest values of $k_q^0$ are achieved when $r$ and $D$ are out of phase ($\omega=1/2$). In other terms, $k_q^0[r;D]$ and $c^*_q[r;D]$ tend to increase with $\omega \in (0,0.5)$. However, this effect is mild for $q=0$, as $k_0^0[r;D]$ appears to be  almost  independent of~$\omega$   (with only minor numerical variations observed).  
The amplitude of the effect of $\omega$ seems to increase with $q$.  This can be interpreted as follows: when $q>0$, the term  $q \partial_x \left( D'(x) u \right)$ in~\eqref{eq:main_KPP_bis} causes individuals to concentrate around regions where $D(x)$ is minimal. This concentration can enhance persistence ($k_q^0[r;D]$) if the minimum of $D(x)$ coincides with the maximum of $r(x)$ ($\omega=1/2$).  When $q$ is large, this effect is amplified, as the advection term dominates, further concentrating individuals in favorable regions.

For $q<0$, such as $q = -1$ and $q = -1/2$, the behavior changes significantly. In these cases, $k_q^0$ and $c^*_q$ tend to decrease with $\omega \in (0,0.5)$. Therefore, the highest speeds and highest values of $k_q^0$ are achieved when $r$ and $D$ are in phase ($\omega=0$). This time, the term  $q \partial_x \left( D'(x) u \right)$ in~\eqref{eq:main_KPP_bis} causes individuals to concentrate in regions where $D(x)$ is maximal, allowing them to better exploit $r(x)$ when its maximum coincides with that of $D$.

\paragraph{Effect of $q$.} We observe that, for each fixed value of $\omega$, the function $q \mapsto k_q^0$ is monotonic (Fig.~\ref{fig:speed_vs_omega}, left). This behavior was predicted by Theorem~\ref{thm:monotonicity}~\emph{(ii)} in the cases $\omega = 0$ and $\omega = 1/2$. Interestingly, and in agreement with the result of Theorem~\ref{thm:monotonicity}~\emph{(iv)}, this monotonicity property no longer holds for $q \mapsto c^*_q$, even in the cases $\omega = 0$ and $\omega = 1/2$.

\paragraph{Optimization of the diffusion coefficient.} When $r$ and $D$ are in phase ($\omega=0$), Fickian diffusion results in higher speeds than other diffusion models (Fig.~\ref{fig:speed_vs_omega}, right). Conversely,  when $r$ and $D$ are out of phase ($\omega=1/2$), Fokker-Planck diffusion leads to the fastest speeds, compared to $q=0$ and even $q=1/2$. To gain insight into the specific interactions between $r$ and $D$ that favor either Fickian or Fokker-Planck diffusion relative to each other, we conducted an optimization procedure. Specifically, we set $r(x) =\cos^2(\pi x)$ and optimized the ratio $c_0^*[r,D] / c_1^*[r,D]$ (respectively, $c_1^*[r,D] / c_0^*[r,D]$) with respect to $D$.
The optimization is performed using a simulated annealing algorithm, which iteratively modifies the function $D$. This function is constructed as a smooth, periodic function using cubic splines, based on a vector of four points whose values lie within the range $(0.1, 1)$. These points correspond to the positions $x = 0$, $x = 1/3$, $x = 2/3$, and $x = 1$, with the first and last values necessarily being identical to ensure periodicity. The optimization process directly operates on this vector of four points to determine the optimal form of $D$. 
See Fig.~\ref{fig:optim}. We observe that the function $D$ maximizing the ratio $c_0^*[r,D] / c_1^*[r,D]$ is in phase with $r$, yielding a ratio of approximately $1.43$. Conversely, the function $D$ that maximizes the ratio $c_1^*[r,D] / c_0^*[r,D]$ is out of phase with $r$, resulting in a ratio of approximately $1.15$.

\begin{figure}
\begin{center}
\includegraphics[scale=0.5]{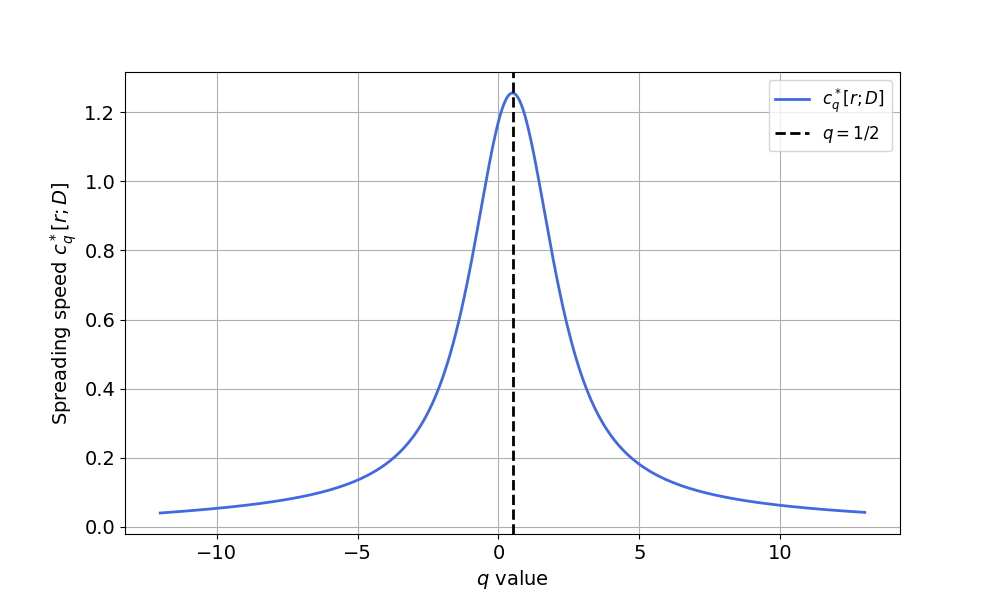}
\end{center}
\caption{Spreading speed $c^*_q[r;D]$ depending on $q$, with a constant growth term $r\equiv 1.$ Here, $D(x)=0.1+ \cos^2(\pi \, x).$}
\label{fig:speed_vs_q}
\end{figure}

\begin{figure}
\begin{center}
\includegraphics[scale=0.3]{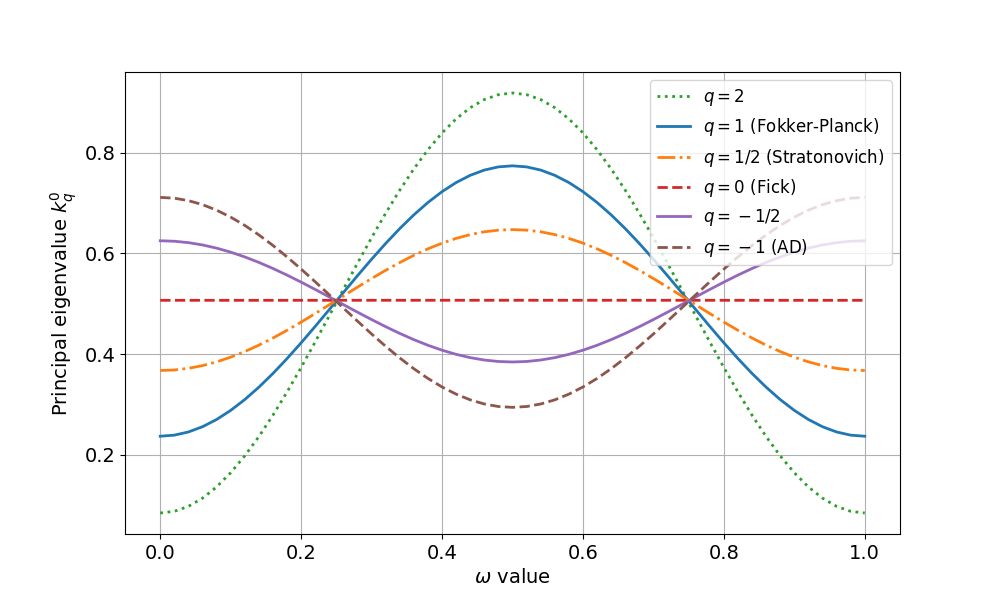}
\includegraphics[scale=0.3]{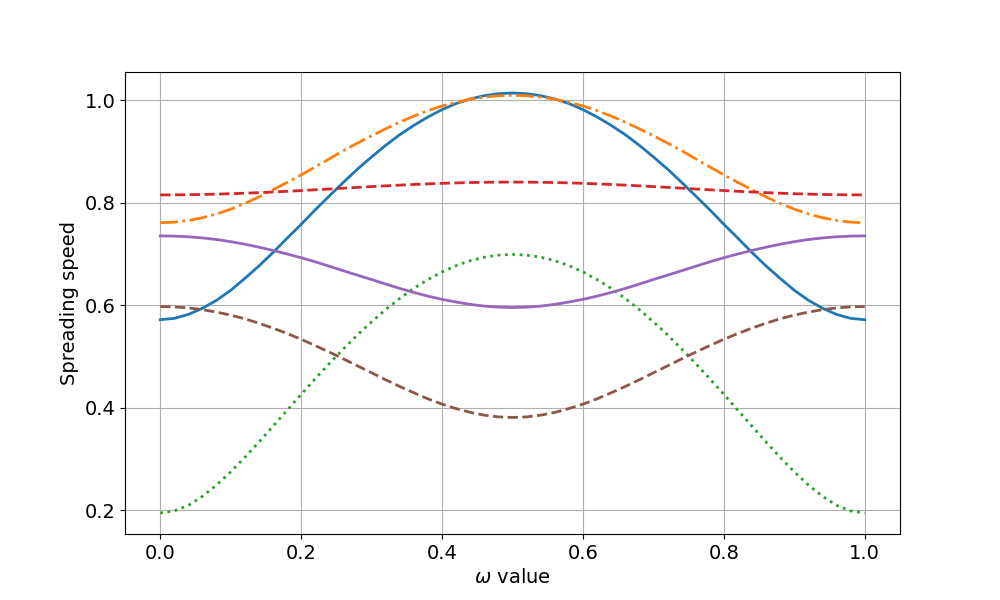}
\end{center}
\caption{Principal eigenvalue $k_q^0$ and spreading speed $c^*_q[r;D]$ depending on the phase shift $\omega$ between $r$ and $D$. Here, $r(x)=\cos^2(\pi \, x)$ and $D(x)=0.1+ \cos^2(\pi \, (x+\omega))$.}
\label{fig:speed_vs_omega}
\end{figure}

\begin{figure}
\begin{center}
\includegraphics[scale=0.5]{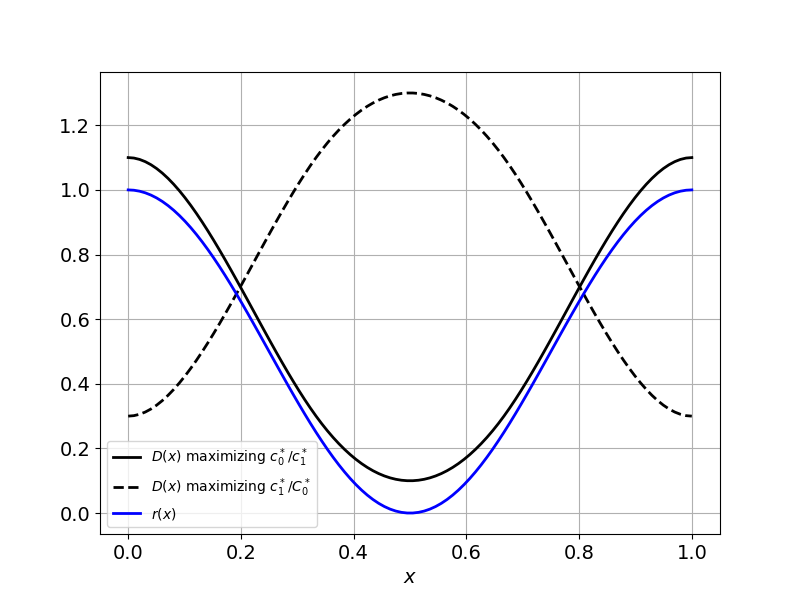}
\end{center}
\caption{Functions $D$ optimizing the ratio $c_0^*[r,D] / c_1^*[r,D]$ (respectively, $c_1^*[r,D] / c_0^*[r,D]$) with $r$ fixed to $r(x)=\cos^2(\pi \, x)$. The optimization is based on a simulated annealing algorithm. The function $D$ is interpolated from four discrete points using cubic splines.}
\label{fig:optim}
\end{figure}

\section{Discussion}\label{s:discussion}

\paragraph{General conclusions.} The comparative analysis of Fickian, Fokker-Planck and other $q$-diffusion models in this study shows that
the choice between these diffusion models is not merely a matter of mathematical preference,
but has profound implications for the understanding of spatial dynamics in heterogeneous environments.
Our investigation underscores the sensitivity of both persistence and spreading properties to the diffusion model employed,
and that there is no general rule for the comparative behavior of persistence and spreading speeds for different values of $q$ across all settings.
In particular, Theorem~\ref{thm:choose_r_d} shows that for any $q\neq0$, suitable configurations of growth rates $r$ and diffusion coefficients $D$ can be found such that the $q$-diffusion either increases or decreases the ability of persistence and the spreading speed with respect to the traditional Fickian diffusion.

\paragraph{The case of constant $r$.} A somewhat unexpected finding emerges when the growth rate $r$ is constant. Despite the heterogeneity of $D$, Fickian ($q=0$) and Fokker-Planck ($q=1$) diffusion models yield equal spreading speeds. 
This result is a consequence of an underlying symmetry in the function $q \mapsto c^*_q$ around $q=1/2$; Fickian and Fokker-Planck diffusions represent a special case of symmetry. 
Moreover, the maximal spreading speed is achieved with Stratonovich diffusion ($q=1/2$). We derive an explicit formula for the maximal spreading speed, and we point out that this maximal spreading speed is equal to the already-known limit of the spreading speed for Fickian diffusion as the period of~$D$ goes to infinity (in this work, the period is mostly kept equal to~$1$, but the same properties hold for any period, as long as~$r$ and~$D$ have the same period).
Furthermore, we prove that $c_{1/2}^*$ does not depend on the period of $D$ when $r$ is constant.

Although the spreading speeds are equal for Fickian and Fokker-Planck diffusion, it is important to note that the shapes of the solutions can differ significantly. For instance, if~$f(x,u) = u(1-u)$ in~\eqref{eq:main_KPP}, the positive equilibrium solution is $u \equiv 1$ for Fickian diffusion, whereas it is necessarily non-constant with Fokker-Planck diffusion if~$D$ is not constant.

\paragraph{Effects of the interactions between $r$ and $D$.} Our analytical results on persistence (Proposition~\ref{prop:largeD}, Theorem~\ref{thm:limits} $(i)$, Theorem~\ref{thm:monotonicity} $(ii)$) and numerical findings on persistence and spreading (Figs.~\ref{fig:speed_vs_omega} and~\ref{fig:optim})
show that, for $q \geq 0$, the maximum ability of persistence and the maximum spreading speed are achieved when $r$ and $D$ are out of phase (i.e., $r$ is large where $D$ is small).  
Biologically, this is relevant and consistent with the existing literature (\cite{KinKaw06}, case $q = 0$): organisms tend 
to remain in favorable environments and leave unfavorable ones. Even when the direction toward a better environment is unknown, because the individuals have access only to very local information, organisms still recognize the necessity of leaving unfavorable areas (this situation could correspond to $q=1/2$ in the case of variable jump length or $q=1$ for variable jump rate, as discussed below).

 On the other hand, our results reveal reversed outcomes regarding phase shift effects when $q<0$. 
Moreover, the amplitude of speed variation with respect to phase shifts is significantly more pronounced in the Fokker-Planck case ($q = 1$) than in the Fickian case ($q = 0$), highlighting important ecological implications. Specifically, in the context of ecological or epidemiological modeling, the strategic placement of dispersal barriers, which reduce $D$ in regions of high or low $r$, is likely to have a much stronger impact under Fokker-Planck diffusion than under Fickian diffusion. More precisely, in order to prevent or slow down an invasion, our results indicate that dispersal barriers should be placed where $r$ is small when $q > 0$, where $r$ is large when $q < 0$, and that the position of these barriers will have little effect when $q$ is close to $0$.

\paragraph{How to choose $q$ in real-world applications?} 
While this question goes beyond the scope of this paper and has already been partly addressed in \cite{AlfGil22,Vis08}, we recall that, as mentioned in the Introduction, when the $q$-diffusion is derived from a microscopic probabilistic model, the value of $q$ depends on two factors: the influence of spatial heterogeneity (e.g.,  the jump rate in \cite{Oku86} or the jump length in \cite{Vis08}) and the chosen reference point. Here, for a jump between $X_t$ and $X_{t+\delta t}$, the reference point is the point~$Y_t$ on the line going through $X_t$ and $X_{t+\delta t}$ where the heterogeneity $D(Y_t)$ (or $\sigma(Y_t)$, following the notations in the Introduction) is evaluated.

Thus, while a middle reference point (i.e.\ $Y_t=(X_t+X_{t+\delta})/2$) leads to $q=1/2$ in the construction based on jump length, the same middle reference point leads to $q=0$ when the construction is based on jump rate~\cite{Oku86}. In practice, depending on the physical or biological system under consideration, the most relevant reference point may vary according to the factor being considered. For example, in some biological contexts where individuals have access only to local information, the jump rate should depend on the starting point, leading to $q=1$; on the other hand, the jump length might depend on the information they gather during their movement and thus depend, e.g., on the midpoint, leading to $q=1/2$.

In real diffusion phenomena, the effective value of~$q$ could be estimated directly from observation data, using a PDE with $q$-diffusion combined with a statistical estimation procedure, as in~\cite{KimLee24}. 
In this case, assuming that the reference point associated with each type of factor is known, estimating the value of $q$ can provide insight into which factor is influenced by spatial heterogeneity. It should be noted that the value of $q$ may also result from a weighted contribution of several factors, each with a different reference point. For example, in \cite{KimLee24}, the authors find an experimental value of $q \approx 0.25$, which they interpret as a predominant effect of factors leading to $q=0$, compared to those associated with larger values of $q$. 

Regarding this last point, let us highlight that if $q \in [-1,1]$, $q$-diffusion can be obtained as a \enquote{convex combination} of the standard models discussed in the Introduction, where $q$ takes values in $\{-1, 0, 1/2, 1\}$. This can be demonstrated by using convex combinations of Equation~\eqref{eq:main_KPP_bis} with $q \in \{-1, 0, 1/2, 1\}$. Consequently, a $q$-diffusion mechanism can also arise from a combination of several standard $q_i$-diffusion mechanisms, where $q_i \in \{-1, 0, 1/2, 1\}$. 
In this paper, we also consider cases where $q \notin [-1,1]$, which fall outside the scope of all standard models but yield interesting mathematical properties.

\paragraph{Extensions.} While our results are derived in a 1D context, extending this analysis to higher spatial dimensions could unveil additional phenomena that further influence the effects of phase shifts on the spreading speed. A possible extension of $q-$diffusion operators to higher dimensions $n\ge 2$ could incorporate different types of diffusion along each spatial direction, allowing for direction-dependent exponents $q_i$. The resulting operator would take the form:
\[
\sum_{i=1}^n \p_{x_i} \left( D_i(x_i)^{1 - q_i} \p_{x_i} \big(D_i(x_i)^{q_i} u \big) \right),
\]
where $D_i(x_i)$ represents the diffusion coefficient along the $i$-th direction, and $q_i$ determines the type of diffusion in that direction. In the context of eco-evolutionary models, populations can be structured both in space and phenotypic traits, leading to heterogeneity arising from entirely different mechanisms. This heterogeneity could, for instance, result from environmental variations in space or genomics variations affecting the mutation rate, making the use of such a generalized model potentially relevant.

\

In summary, this work advocates for a paradigm shift in the conventional use of Fickian diffusion in mathematical biology, emphasizing the deliberate and informed selection of diffusion models tailored to the unique attributes of the system under study. It underscores the importance of accounting for the interplay between diffusion dynamics and spatial heterogeneity in designing theoretical models, often masked by the homogenizing effect of Fickian diffusion.

\section{Proofs and technical lemmas}\label{s:proofs}

\subsection{Proof of Theorem \ref{thm:choose_r_d}, persistence}

\paragraph{Proof of Theorem~\ref{thm:choose_r_d}, \emph{(i)} and \emph{(ii)}, persistence.}
We first prove Proposition~\ref{prop:compar_k}. We will then use it to prove the persistence part of Theorem~\ref{thm:choose_r_d}, \emph{(i)} and \emph{(ii)}.

\begin{proof}[Proof of Proposition~\ref{prop:compar_k}]

  Let $\psi>0$ be a principal eigenfunction associated with \allowbreak $k_0^\lambda[r;D]$. By definition, we have
\begin{equation}\label{eq:vp_0}
    D \psi'' + [ D' -2 \lambda D] \psi' + [ \lambda^2 D - \lambda D' + r] \psi = k_0^\lambda[r;D] \psi.
\end{equation}
Let $\gamma>0$ and set
\[\phi := \frac{\psi}{D^{\gamma}}.\]
Replacing $\psi$ with $\phi D^\gamma$ in \eqref{eq:vp_0} and dividing the equation by $D^\gamma$, a computation gives 
 \begin{multline*}
   D\phi'' + [ (1+2\gamma) D' -2 \lambda D] \phi' + \cro{ \gamma D'' + \gamma^2 \frac{(D')^2}{D} + \lambda^2 D -(1+ 2\gamma) \lambda D' + r } \phi \\
   =k_0^\lambda[r;D] \phi. 
 \end{multline*}
Taking $\gamma=q/2$, we get: 
\begin{equation*}
    \Ll_{q}^\lambda\cro{ r- \frac{q}{2}D'' +\frac{q^2}{4} \frac{(D')^2}{D} ; D } \phi =  k_0^\lambda[ r ; D ] \phi.
\end{equation*}
The  uniqueness of the principal eigenvalue implies that
\[ k_q^\lambda\cro{ r- \frac{q}{2}D'' +\frac{q^2}{4} \frac{(D')^2}{D} ; D }= k_0^\lambda[r;D] ,\]
which in turn leads to the result of the proposition.

\end{proof}

Next, we use the result of Proposition~\ref{prop:compar_k} to exhibit a situation where Fickian diffusion ($q=0$) increases the ability of persistence.

\begin{lemma}[Theorem~\ref{thm:choose_r_d}, \emph{(i)}, persistence] \label{lem:Fick_persistr1}
  Assume that $D$ is not constant. Let $r_0\in\R$ be a constant and, for $q\in\R\setminus\acc{0}$, let
  $\ds h_q:=-\frac{q}{2}D'' +\frac{q^2}{4} \frac{(D')^2}{D}$.\\
  For all $q\in \R \setminus \{0\},$ we have, with $r(x):=r_0 + h_q(x)$,
    \[k_q^0[r;D] < k_0^0[r;D].\]
\end{lemma}

\begin{proof} 

  Using Proposition~\ref{prop:compar_k}, we have 
\begin{equation} \label{eq:k0h}
    k_q^0[r;D]=k_0^0[ r-h_q; D ] = k_0^0[ r_0; D ]. 
\end{equation}
Let \[\delta:= \int_0^1 h_q.\]
Since $D$ is not constant, we have $\delta >0$.
Hence
\begin{equation} \label{eq:k0delta}
k_0^0[ r_0; D ] = r_0 < r_0+\delta = k_0^0[ r_0+\delta; D ].     
\end{equation}
Let $\psi$ a principal eigenfunction associated with $k^{0}_0[ r;D ]$. By definition, we have
\begin{equation*}
     (D\psi')' +  r \psi =  k^{0}_0[ r;D ] \psi.
\end{equation*}
Dividing this equation by $\psi$ and integrating by parts, we get:
\begin{equation} \label{eq:divipp}
     \int_0^1 D \frac{(\psi')^2}{\psi^2} +   r_0 + \delta =  k^{0}_0[ r;D ],
\end{equation}
which implies that  $ k^{0}_0[ r;D ] \ge r_0 + \delta.$ Using this inequality with \eqref{eq:k0h} and \eqref{eq:k0delta}, we get $k^{0}_q[ r;D ] <k^{0}_0[ r;D ]$.

\end{proof}

We then use the fact that $-h_q$ takes some positive values to show that the reverse inequality can be true for another choice of $r$.
\begin{lemma}[Theorem~\ref{thm:choose_r_d}, \emph{$(ii)$}, persistence] \label{lem:Fick_persistr2}
    Assume that $D$ is not constant. For $q\in\R\setminus\acc{0}$, let
  $\ds h_q:=-\frac{q}{2}D'' +\frac{q^2}{4} \frac{(D')^2}{D}$.\\
  For all $q\in \R \setminus \{0\},$ we have, with $r(x):=-a h_q(x)$, and $a>0$ large enough,
    \[k_q^0[r;D] > k_0^0[r;D].\]
\end{lemma}

\begin{proof} 
  Set $g(a):=k_0^0[ - a \, h_q ; D ],$ for $a \ge 0.$
  It follows from the periodicity of $D$ that there exists $x_0$ in $[0,1]$ such that $-h_q(x_0)>0$.
  Thus, Proposition~5.2 in \cite{BerHamRoq05a} also implies that there exists $a_0>0$ such that for $a>a_0$, $g(a)>0$ and $g$ is increasing on $[a_0,+\infty)$.
    In particular, $g(1+a)> g(a)$ for $a>a_0$.
    Coming back to the definition of $g$, we have $k_0^0[ - a \, h_q - h_q ; D ] > k_0^0[ - a \, h_q  ; D ]$ for $a>a_0$.
    Using Proposition~\ref{prop:compar_k}, we get: $$k_q^0[ - a \, h_q ; D ] > k_0^0[ - a \, h_q  ; D ].$$
\end{proof}

\paragraph{Proof of Theorem~\ref{thm:choose_r_d}, \emph{(iii)} and \emph{(iv)}.}
As mentioned above, the main ingredient of the proof of \emph{(iii)} and \emph{(iv)} in Theorem~\ref{thm:choose_r_d} is Proposition~\ref{prop:largeD}. We first prove a variational formula for $k^0_q[r;D]$ as an intermediate step for the proof of Proposition~\ref{prop:largeD}.

\begin{lemma}\label{lem:variational}
  We have:
  \begin{equation} \label{eq:var1}
    k^{0}_q[r;D] =\max_{\varphi \in E } \pth{-\int_0^1 D^{1-q} |(D^{q/2} \varphi)'|^2 + \int_0^1 r \varphi^2},
  \end{equation}
where $E$ is the set of positive $1-$periodic $\mcc^2$ functions $\varphi$ with $\int_0^1  \varphi^2 =1$.
\end{lemma}

\begin{proof}

  Let $\psi>0$ be a principal eigenfunction associated with $k^{0}_q$. By definition, we have
\begin{equation*}
     (D^{1-q} (D^q \psi)')' +  r \psi =  k^{0}_q \psi.
\end{equation*}
Setting $v=D^q \, \psi$, we get:
\begin{equation} \label{eq:vpv}
     (D^{1-q} v')' +  \frac{r}{D^q} v =  \frac{k^{0}_q}{D^q} v.
\end{equation}
Since the operator $\mcl :v\mapsto (D^{1-q} v')' +  \frac{r}{D^q} v$ is self-adjoint, the operator $\mct:\varphi\mapsto D^{q/2}\mcl\pth{D^{q/2}\varphi}$ is self-adjoint as well.
The principal eigenvalue of $\mct$ is precisely $k^0_q$, and is given by a Rayleigh quotient:
	\begin{equation} \label{eq:var1p}
k^{0}_q=\max_{\phi \in G } \frac{-\int_0^1 D^{1-q}(\phi')^2 + \int_0^1 r D^{-q} \phi^2}{\int_0^1 D^{-q} \phi^2}
	\end{equation}
where~$G$ is the set of positive smooth $1-$periodic functions. Up to multiplication of $\phi$ by a constant in \eqref{eq:var1p}, we may assume that $\int_0^1 D^{-q} \phi^2=1$, and setting $\varphi= D^{-q/2} \phi$, this concludes the proof of Lemma~\ref{lem:variational}.
\end{proof}

We are now ready to prove Proposition~\ref{prop:largeD}.

\begin{proof}[Proof of Proposition~\ref{prop:largeD}]
  First,  we have $k^{0}_q[ r; B\, D ] \le    \max_{[0,1]} r$.
  Second, 
    let
\[\langle D^q \rangle_H := \frac{1}{\int_0^1 1/D^{q}} \]
be the harmonic mean of $D^q$.
Take 
\[\varphi = \sqrt{\langle D^q \rangle_H } D^{-q/2}\]
in the variational formula~\eqref{eq:var1} of Lemma~\ref{lem:variational}. This gives
\begin{equation} \label{eq:bounds_k^0}
   \langle D^q \rangle_H \int_0^1 \frac{r}{D^q} \le k^{0}_q[ r; B\, D ] \le    \max_{[0,1]} r.
\end{equation}
Now, take an arbitrary sequence $B_{n} \to +\infty$.
Let $\psi_n$ be the principal eigenfunction associated with $k^{0}_q[ r; B_n \, D ]$,
with the normalization condition $\| \psi_n\|_{L^2(0,1)}=1$.
By definition, we have
\begin{equation} \label{eq:psi_n}
     (D^{1-q}(D^q \psi_n)')' +  \frac{r}{B_n} \psi_n =  \frac{k^{0}_q[ r; B_n \, D ] }{B_n}\psi_n.
\end{equation}
Since
$k^{0}_q[ r; B_n \, D ] $ is bounded (see~\eqref{eq:bounds_k^0}), standard elliptic
estimates and Sobolev injections imply that, up to extraction of
a subsequence, the sequence of functions $\psi_{n}$ converges to a
nonnegative function $\psi$, locally in $W^{2,p}$ for all $1<p<\infty$.
Furthermore,
$\| \psi\|_{L^2(0,1)}=1$, $\psi$ is periodic and satisfies
\begin{equation*}
 (D^{1-q}(D^q \psi)')'=0.
\end{equation*}
Thus $\psi\equiv C\, D^{-q}$, for some positive constant $C$. Since $\| \psi\|_{L^2(0,1)}=1$, we have $C \| D^{-q}\|_{L^2(0,1)} =1$, thus $C=\sqrt{\langle D^{2 \, q} \rangle_H }$. 

Finally, we set $v_n:=D^q \, \psi_n$. We have:
\begin{equation} \label{eq:vpvn}
     B_n \, (D^{1-q} v_n')' +  \frac{r}{D^q} v_n =  \frac{k^{0}_q[ r; B_n \, D ] }{D^q} v_n.
\end{equation}
We multiply by $v_n$ and integrate by parts. We get:
\begin{equation} \label{eq:vpvn2}
    -B_n \int_0^1 D^{1-q}(v_n')^2 +  \int_0^1 \frac{r}{D^q} v_n^2 =  k^{0}_q[ r; B_n \, D ] \int_0^1 \frac{v_n^2}{D^q},
\end{equation}
which implies that 
\begin{equation*}
     k^{0}_q[ r; B_n \, D ] \le \frac{ \ds \int_0^1 \frac{r}{D^q} v_n^2}{\ds \int_0^1 \frac{v_n^2}{D^q}}.
\end{equation*}
Using \eqref{eq:bounds_k^0} and passing to the limit $n\to \infty$, we get $v_n \to \sqrt{\langle D^{ 2 q }\rangle_H }$ in $L^2(0,1)$
and
\[\lim_{n\to \infty}k^{0}_q[ r; B_n \, D ] =  \langle D^q \rangle_H \int_0^1 \frac{r}{D^q}. \]
The family $(k^{0}_q[ r; B \, D ])_{B>0}$ is bounded, and there is only one possible limit to a converging subsequence.
Therefore, the whole family converges to the same limit, namely:
\[\lim_{B\to \infty}k^{0}_q[ r; B \, D ] =  \langle D^q \rangle_H \int_0^1 \frac{r}{D^q}. \]

\end{proof}

\begin{lemma}[Theorem~\ref{thm:choose_r_d}, \emph{(iii)} and \emph{(iv)}]\label{lem:inegk^0}
  Let $r$ be fixed.
  Assume  that $r$ is nonconstant.
  For all $q\in \R \setminus \acc{0}$,
  there exists a $1-$periodic~$D_0$ such that $k^{0}_0[ r; D_0 ]> k^{0}_q[ r; D_0 ]$, and there exists a $1-$periodic~$D_1$ such that $k^{0}_0[ r; D_1 ]< k^{0}_q[ r; D_1 ]$. 
\end{lemma} 

\begin{proof}
  For all $q\in\R$ and $\kappa\in\R$, we have  $k^0_q[r+\kappa;D]=k^0_q[r;D]+\kappa$. 
    Therefore, up to replacing~$r$ with $r+\kappa$ with $\kappa$ large enough, we may assume that~$r>0$ without changing the ordering of the eigenvalues.
  
  Consider $D := r^{1/q}$ and let $D_0 = B \, D$.
  Proposition~\ref{prop:largeD} implies that
  \[\lim_{B\to \infty}k^{0}_0[ r; B\, D ]=  \int_0^1r\]
  and
  \[\lim_{B\to \infty}k^{0}_q[ r; B\, D ]=  \pth{\int_0^1\frac{1}{r}}^{-1}.\]
  Jensen's inequality implies that ($r$ being not constant)
  \[\int_0^1r>\pth{\int_0^1\frac{1}{r}}^{-1}.\]
  Therefore, for $B$ large enough, $k^{0}_0[ r; D_0 ] >  k^{0}_q[ r; D_0 ]$.

  To prove the other inequality, let $r_{max}$ be the maximum of $r$ over $[0,1]$. Assume that the maximum is reached at some point $x_m \in [0,1]$. Take $\varepsilon>0$ such that $r_{max}-2 \varepsilon> \int_0^1 r$ and choose $D$ such that $D^{-q}/\int_0^1 D^{-q}$ is \enquote{close to a Dirac mass at $x_m$}, in the sense that:
  \[ \int_0^1 r \frac{D^{-q}}{\int_0^1 D^{-q}}> r_{max} -\varepsilon.\]
  Proposition~\ref{prop:largeD} implies that
  \[\lim_{B\to \infty}k^{0}_0[ r; B\, D ]=\int_0^1 r \]
  and
  \[\lim_{B\to \infty}k^{0}_q[ r; B\, D ]=
  \int_0^1 r \frac{D^{-q}}{\int_0^1 D^{-q}}> r_{max} -\varepsilon>
  \int_0^1 r.\]
  This concludes the proof of the second inequality in Lemma~\ref{lem:inegk^0}, with $D_1= B\, D$ and~$B$ large enough. 

\end{proof}

\subsection{Proof of Theorem \ref{thm:choose_r_d}, spreading}

We now finish the proof of Theorem~\ref{thm:choose_r_d}.
Namely, we prove that if~$q$ and~$D$ are fixed, then there exists~$r$ such that $c_q^*[r;D]>c_0^*[r;D]$ and another~$r$ such that $c_q^*[r;D]<c_0^*[r;D]$.

We first prove Proposition~\ref{prop:compar_persist_speed}, which allows us to extend a comparison on the ability of persistence to a comparison on the spreading speed (when the ability of persistence is close to $0$).
Using the persistence part of Theorem~\ref{thm:choose_r_d}, we will then be able to prove the spreading part.

\begin{proof}[Proof of Proposition~\ref{prop:compar_persist_speed}]
We first show that $ c_{\tilde{q}}^*[ r - \kappa; D ] \to 0$ as $\kappa \to k_{\tilde{q}}^0[r;D].$ 
Assume by contradiction that there exist $\delta>0$ and a sequence $\kappa_n \in (0,k_{\tilde{q}}^0[r;D])$ such that $\kappa_n \to k_{\tilde{q}}^0[r;D]$ as $n\to \infty$ and $c_{\tilde{q}}^*[ r - \kappa_n; D ] > \delta$ for all $n$.

From the definition of $k_{\tilde{q}}^\lambda$, by uniqueness, we have $k_{\tilde{q}}^\lambda[ r - \kappa_n; D ]=k_{\tilde{q}}^\lambda[ r ; D ]- \kappa_n,$
and consequently,
\[ c_{\tilde{q}}^*[ r - \kappa_n; D ] =\inf_{\lambda>0} \left(\frac{k_{\tilde{q}}^\lambda[ r ; D ]}{\lambda} - \frac{\kappa_n}{\lambda}\right).\]
Thus,
\[ \frac{k_{\tilde{q}}^\lambda[ r ; D ]}{\lambda} - \frac{\kappa_n}{\lambda}\ge c_{\tilde{q}}^*[ r - \kappa_n; D ] \ge \delta\]
for all $n\ge 0$ and $\lambda>0$.
Passing to the limit $n\to \infty$, we get:
\begin{equation} \label{eq:lambdato0}
\forall\lambda\in\R,\qquad    \ds \frac{k_{\tilde{q}}^\lambda[ r ; D ]}{\lambda} - \frac{k_{\tilde{q}}^0[ r ; D ]}{\lambda}\ge \delta.
\end{equation}
We know that the function $\lambda \mapsto k_{\tilde{q}}^\lambda[r;D]$ is analytic (see e.g. \cite{Nad09a}) and that $\lambda=0$ is a global minimum of this function (using the convexity and Proposition~\ref{prop:left_right}), thus we have $\left. \p_\lambda k_{\tilde{q}}^\lambda[r;D] \right|_{\lambda=0}= 0.$ On the other hand, passing to the limit $\lambda \to 0$ in \eqref{eq:lambdato0}, we get $\left. \p_\lambda k_{\tilde{q}}^\lambda[r;D] \right|_{\lambda=0} \ge \delta$. Thus we get a contradiction. 
This proves that
\[c_{\tilde{q}}^*[ r - \kappa; D ] \to 0
\text{ as }\kappa \to k_{\tilde{q}}^0[r;D].\]
On the other hand, $k_{q}^\lambda[ r -\kappa ; D ] = k_{q}^\lambda[ r  ; D ]-\kappa$ is a nonincreasing function of $\kappa$
and therefore $ c_{q}^*[ r - \kappa; D ] $ is also a nonincreasing function of $\kappa$.
Additionally, 
\[k_{q}^\lambda[ r - k_{\tilde{q}}^0[ r  ; D ] ; D ]
= k_{q}^\lambda[ r  ; D ]-k_{\tilde{q}}^0[ r  ; D ]
\ge k_{q}^0 [ r  ; D ] -  k_{\tilde{q}}^0[ r  ; D ]>0\]
for all $\lambda>0$.
Thus, $ c_{q}^*[ r - \kappa; D ] \ge  c_{q}^*[ r - k_{\tilde{q}}^0[ r  ; D ] ; D ] >0$ for all  $\kappa \in [0,k_{\tilde{q}}^0[ r  ; D ]]$. As 
we have shown that  $ c_{\tilde{q}}^*[ r - \kappa; D ] \to 0$ as $\kappa \to k_{\tilde{q}}^0[r;D],$ this concludes the proof.

\end{proof}

\begin{proof}[Proof of Theorem~\ref{thm:choose_r_d}, $(i)-(ii)$, spreading speed]
Let~$q$ and~$D$ be fixed.
  Lemma~\ref{lem:Fick_persistr1} and Proposition~\ref{prop:compar_persist_speed} show the existence of a $1-$periodic $r\in\mcc^{0,\alpha}(\R)$ such that $c_q^*[r;D] < c_0^*[r;D]$. 
  Conversely, Lemma~\ref{lem:Fick_persistr2} and Proposition~\ref{prop:compar_persist_speed} show the existence of a $1-$periodic $r\in\mcc^{0,\alpha}(\R)$ such that $c_q^*[r;D] > c_0^*[r;D]$.
\end{proof}

\subsection{Deformation of space and proof of Theorem \ref{thm:r_constant}}

The goal of this section is to prove Theorem \ref{thm:r_constant}.
The main ingredients of the proof of Theorem \ref{thm:r_constant} are Propositions~\ref{ppn:equiv_unif_laplacian} and~\ref{prop:VP}.
Using a deformation of space, we first prove the following lemma. As a consequence, we will prove Proposition~\ref{ppn:equiv_unif_laplacian}.
Next, using Proposition~\ref{prop:VP} from~\cite{BouHamRoq25} and Proposition~\ref{prop:left_right}, we will conclude the proof of Theorem~\ref{thm:r_constant}.

\begin{lemma}\label{lem:space_deformation}
Let $D\in\mcc^{2,\alpha}(\R)$, $b\in\mcc^{1,\alpha}(\R)$ and $r\in\mcc^{0,\alpha}(\R)$ be $1-$periodic functions, with $D>0$.
Consider the diffeomorphism
\[h(x):=\int_0^x\frac{\de x'}{\sqrt{D(x')}}.\]
Let 
\[B(y):=\frac{2b+D'}{2\sqrt{D}}\pth{h^{-1}(y)}.\]
Let $v\geq 0$ solve the linear evolution problem
\[\dr_tv=\dr_{xx}(D(x)v)+\dr_x(b(x)v)+r(x)v,\]
with a nonnegative and locally bounded initial condition~$v(0,\cdot)$.
Let
\[V(t,y):={ \sqrt{D(h^{-1}(y))}}\,v\pth{t,h^{-1}(y)}\]
and $R(y):=r\pth{h^{-1}(y)}$. Then $V$ solves the linear evolution problem
\[\dr_tV=\dr_{yy}V+\dr_y(B(y)V)+R(y)V.\]
\end{lemma}

\begin{proof}[Proof of Lemma~\ref{lem:space_deformation} when $r\equiv0$]
  First, we assume that $r\equiv 0$ and that the initial data~$v(0,\cdot)$ is integrable with $\int_0^1v(0,\cdot)=1$.
  These assumptions allow us to understand the lemma from a probabilistic point of view.
  The proof for general~$r$ is much more technical and will be done below.

  We use the interpretation of the conservative evolution equation as the Fokker-Planck equation for an Itô diffusion.
  (Next, we will use linearity to prove the result for initial data that are not integrable.)
  
 By our assumptions, for all $t\geq0$, $v(t,\cdot)$ is the density
of the law of a particle satisfying the SDE
\[\de X_t=-b(X_t)\de t+\sqrt{2D(X_t)}\de W_t\]
with initial law of density $v(0,\cdot)$ (for some Brownian motion $(W_t)_t$).
Consider the change of variables $Y_t=h(X_t)$. Then Itô's formula gives
\begin{align*}
  \de Y_t &= -h'(X_t)\times b(X_t)\de t+h'(X_t)\sqrt{2D(X_t)}\de W_t+\frac{1}{2}h''(X_t)\times 2D(X_t)\de t\\
  &=\cro{-\frac{b}{\sqrt{D}}-\frac{D'}{2\sqrt{D}}}(X_t)\de t+\sqrt{2}\de W_t\\
  &=-\frac{2b+D'}{2\sqrt{D}}(X_t)\de t+\sqrt{2}\de W_t.
\end{align*}
The Fokker-Planck equation for this Itô diffusion is precisely:
\[\dr_t V(t,y) = \dr_{yy}V+\dr_y\pth{B(y)V}.\]
This proves the result for $r\equiv0$ and $\int v(0,\cdot)=1$.
Finally, using the linearity of the equations, we conclude that the result holds for all nonnegative $v(0,\cdot)\in L^1(\R)$. Last, using the fact that solving the equation is a local property, the result holds for all nonnegative and locally bounded~$v(0,\cdot)$.

\end{proof}

\begin{proof}[Proof of Lemma~\ref{lem:space_deformation} for general~$r$]
 When $r\not\equiv0$, the proof is unfortunately much more technical, because we need to do the explicit computations. 
  For conciseness, let us denote $H:=h^{-1}$. Note that $H'(y)=\sqrt{D(H(y))}$.
  First, we compute:
  \begin{align*}
    \dr_yV(t,y)&=\sqrt{D(H(y))}H'(y)\dr_yv(t,H(y))+\frac{H'(y)D'( H(y))}{2\sqrt{D(H(y))}}v(t,H(y))\\
      &=D(H(y))\dr_yv(t,H(y))+\frac{D'(H(y))}{2}v(t,H(y)).
  \end{align*}
  Therefore,
  \begin{align*}
    \dr_{yy}V(t,y)
    &=H'(y)D'(H(y))\dr_yv(t,H(y))+H'(y)D(H(y))\dr_{yy}v(t,H(y))\\
    &\qquad+\frac{H'(y)D'(H(y))}{2}\dr_yv(t,H(y)) + \frac{H'(y)D''( H(y))}{2}v(t,H(y))\\
    &=\frac{3}{2}D'(H(y))\sqrt{D(H(y))}\dr_yv(t,H(y))+D(H(y))^{3/2}\dr_{yy}v(t,H(y))\\
    &\qquad+\sqrt{D(H(y))}\frac{D''( H(y))}{2}v(t,H(y)).
  \end{align*}
  Next, we develop the expression of $\dr_tv$:
  \begin{align*}
    \dr_tv(t,x)&=D\dr_{xx}v(t,x)+(b+2D')\dr_xv(t,x)+(r+D''+b')v(t,x).
  \end{align*}
Now, we can compute:
\begin{align*}
  \dr_tV(t,y)&=\sqrt{D(H(y))}\dr_tv(t,H(y))\\
  &=D(H(y))^{3/2}\dr_{yy}v(t,H(y))\\
  &\qquad+\pth{\sqrt{D(H(y))} b+2D'(H(y))\sqrt{D(H(y))}}\dr_yv(t,H(y))\\
  &\qquad+\sqrt{D(H(y))}(r(H(y))+D''(H(y))+b'(H(y)))v(t,H(y)).
\end{align*}
The last expression can be rewritten (each line in the expression below corresponds to a line in the expression above):
\begin{align}
  \dr_tV(t,y)
  &=\dr_{yy}V(t,y)-\frac{3}{2}D'(H(y))\sqrt{D(H(y))}\dr_yv(t,H(y))-\sqrt{D(H(y))}\frac{D''( H(y))}{2}v\nonumber\\
  &\qquad+\pth{\sqrt{D(H(y))} b+2D'(H(y))\sqrt{D(H(y))}}\dr_yv(t,H(y))\nonumber\\
  &\qquad+(R(y)+D''(H(y))+b'(H(y)))V(t,H(y)).\label{eq:lem_deformation_space_1}
\end{align}
The second-order term in~\eqref{eq:lem_deformation_space_1} is $\dr_{yy}V(t,y)$, as required. The first-order term in~\eqref{eq:lem_deformation_space_1} is
\begin{align*}
  &\pth{-\frac{3}{2}D'(H(y))\sqrt{D(H(y))}+\sqrt{D(H(y))} b+2D'(H(y))\sqrt{D(H(y))}}\dr_yv(t,H(y))\\
  &\qquad=\pth{\frac{1}{2}D'(H(y))\sqrt{D(H(y))}+\sqrt{D(H(y))} b}\frac{\dr_yV(t,y)-\frac{D'(H(y))}{2}v(t,H(y))}{D(H(y))}\\
  &\qquad=B(y)\dr_yV(t,y)-\frac{B(y)D'(H(y))}{2}v(t,H(y)).
\end{align*}
We obtain, recalling that $V(t,y)=\sqrt{D(H(y))}v(t,H(y))$,
\begin{align}
  \dr_tV(t,y)
  &=\dr_{yy}V(t,y)+B(y)\dr_yV(t,y)\nonumber\\
  &\qquad+\pth{-\frac{B(y)D'(H(y))}{2\sqrt{D(H(y))}}+R(y)+\frac{D''(H(y))}{2}+b'(H(y))}V(t,H(y)).\label{eq:lem_deformation_space_2}
\end{align}
To compute the zero-order term in~\eqref{eq:lem_deformation_space_2}, we first compute:
\begin{align*}
  B'(y)&=H'(y)\pth{\frac{b'(H(y))}{\sqrt{D(H(y))}}-\frac{b(H(y))D'(H(y))}{D(H(y))^{3/2}}+\frac{D''(H(y))}{2\sqrt{D(H(y))}}-\frac{D'(H(y))^2}{4D(H(y))^{3/2}}}\\
  &=b'(H(y))-\frac{D'}{2D}b(H(y))+\frac{D''(H(y))}{2}-\frac{D'(H(y))^2}{4D(H(y))}.
\end{align*}
The zero-order term in~\eqref{eq:lem_deformation_space_2} is therefore:
\begin{align*}
  &\pth{-\frac{B(y)D'(H(y))}{2\sqrt{D(H(y))}}+R(y)+\frac{D''(H(y))}{2}+b'(H(y))}V(t,y)\\
  &\qquad=\pth{-\frac{bD'(H(y))}{2D(H(y))}-\frac{D'(H(y))^2}{4D(H(y))}+R(y)+\frac{D''(H(y))}{2}+b'(H(y))}V(t,y)\\
  &\qquad=\pth{R(y)+B'(y)}V(t,y).
\end{align*}
We conclude that
\begin{align*}
  \dr_tV
  &=\dr_{yy}V+B\dr_yV+\pth{R+B'}V\\
  &=\dr_{yy}V+\dr_y(BV)+RV,
\end{align*}
as stated.

%
\end{proof}

\begin{proof}[Proof of Proposition \ref{ppn:equiv_unif_laplacian}]
  In this proof, we write for conciseness $\mcl_q^{\lambda}$ instead of $\mcl_q^{\lambda}[r;D]$ and $k_q^{\lambda}$ instead of $k_q^{\lambda}[r;D]$.
  
  First, we let $\varphi>0$ be a principal eigenfunction of the operator~$\mcl_q^{\lambda}$.
  We have:
  \[(\mcl_q^{\lambda}-k_q^{\lambda})\varphi=0.\]
  Second, we recall that by the definition of $\mcl_q^{\lambda}$, we have, for all $\phi\in\mcc^{2,\alpha}(\R)$,
  \[E_{\lambda}\,\mcl_q^0(E_{-\lambda}\phi)=\mcl_q^{\lambda}\phi
  \qquad\text{ with }E_{\lambda}(x):=e^{\lambda x}.\]
  Therefore,
  \[(\mcl_q^{0}-k_q^{\lambda})(E_{-\lambda}\varphi)=0.\]
  Now, we rewrite $\mcl_q^0$ as:
  \[\mcl_q^0\phi=\dr_x\pth{D^{1-q}\dr_x\pth{D^q\phi}}+r\phi=\dr_{xx}(D(x)\phi)-(1-q)\dr_x\pth{D'\phi}+r\phi.\]
  The function $(t,x)\mapsto E_{-\lambda}(x)\varphi(x)$ solves:
  \[0=\dr_t(E_{-\lambda}\varphi)=(\mcl_q^{0}-k_q^{\lambda})(E_{-\lambda}\varphi).\]
We let $\Phi(y):=\sqrt{D(h^{-1}(y) )}E_{-\lambda}(h^{-1}(y))\varphi(h^{-1}(y))$. 
Then, Lemma~\ref{lem:space_deformation} with $v(t,x)=E_{-\lambda}(x)\varphi(x)$
implies that~$\Phi$ satisfies
\[\dr_{yy}\Phi-s_q\dr_y\pth{\frac{D'}{\sqrt{D}}\pth{h^{-1}(y)}\times \Phi }+(R(y)-k_q^{\lambda})\Phi=0,\]
 with $s_q:=\frac{1}{2}-q$.
Setting $\drift(y):=\ln \pth{D\pth{h^{-1}(y)}}$ gives 
\[\drift'(y)=\frac{D'}{h' D}\pth{h^{-1}(y)}=\frac{D'}{\sqrt{D}}\pth{h^{-1}(y)}.\]
Thus
\begin{equation}\label{eq:unif_laplacianPhi}
  (\widehat{\mcl}_q^{0}-k_q^\lambda)\Phi
  =\dr_{yy}\Phi-s_q\dr_y\pth{\drift'\Phi}+(R(y)-k_q^\lambda)\Phi=0.
\end{equation}
Now, a short computation
shows that
\[E_{\lambda D_H}\widehat{\mcl}_q^0\pth{E_{-\lambda D_H}\phi}=\widehat{\mcl}_q^{\lambda D_H}\phi,\]
where we denote $D_H:=\scal{\sqrt{D}}_H$.
Therefore, writing
$\psi:=E_{{\lambda D_H} }\,\Phi$ and
using~\eqref{eq:unif_laplacianPhi}, we obtain:
\[(\widehat{\mcl}_q^{\lambda D_H}-k_q^\lambda)\psi=0.\]
Furthermore, we note that
\[\psi(y)=e^{\lambda\pth{yD_H-h^{-1}(y)}}\varphi(h^{-1}(y)).\]
Hence, since $D_H=1/h(1)$, we find that $\psi$ is $h(1)-$periodic.
Since we also have $\psi>0$, we conclude that $\psi$ is a principal eigenfunction for $\widehat{\mcl}_q^{\lambda D_H}$ and that the associated principal eigenvalue is $k_q^\lambda$.
Therefore, $\widehat{k}_q^{\lambda D_H}=k_q^\lambda$, which is what we wanted to prove.  
%
\end{proof}

We now prove Proposition~\ref{prop:left_right}, which is the last ingredient of Theorem~\ref{thm:r_constant}.

\begin{proof}[Proof of Proposition~\ref{prop:left_right}]
  
  The case $q=0$ is standard (see e.g. \cite{Nad15}). We prove it for the sake of completeness.
  Let $\psi$ be a principal eigenfunction associated with $k_0^\lambda$, and let $\tpsi$ be a principal eigenfunction associated with $k_0^{-\lambda}.$ We have
\begin{align*}
    & (D \psi')'  -2 \lambda D \psi' + [ \lambda^2 D - \lambda D' + r] \psi = k_0^\lambda \psi, \\
    & (D \tpsi')' + 2 \lambda D \tpsi' + [ \lambda^2 D + \lambda D' + r] \tpsi = k_0^{-\lambda} \tpsi.
\end{align*}
Multiplying the first equation by $\tpsi$ and integrating by parts, we get:
\begin{equation*}
    \int_0^1 [ (D \tpsi')' + 2 \lambda  D \tpsi' + (  \lambda^2 D  +  \lambda D' + r)\tpsi  ] \psi = k_0^\lambda  \int_0^1 \tpsi\psi.
\end{equation*}
Thus,
\[k_0^{-\lambda}  \int_0^1 \tpsi\psi = k_0^\lambda  \int_0^1 \tpsi\psi,\]which implies that $k_0^{-\lambda}[ r ; D]=k_0^{\lambda}[ r ; D]$.

Last, for general $q\in\R$, we use Proposition~\ref{prop:compar_k} and get: 
\[k_q^{-\lambda}[r;D] = k_0^{-\lambda}[ r - h_q ; D ]= k_0^{\lambda}[ r - h_q ; D ]=k_q^{\lambda}[ r  ; D ].\]
  Using the Freidlin-Gärtner formulas (\eqref{eq:garfre} and~\eqref{eq:garfre_left}), we obtain that the spreading speed to the left and the spreading speed to the right coincide:
\[c^{*}_q[r;D]=\inf_{\lambda>0}\frac{k_q^{\lambda}[r;D]}{\lambda}=\inf_{\lambda>0}\frac{k_q^{-\lambda}[r;D]}{\lambda}=c_q^{*,\,left}[r;D].\]

\end{proof}

\begin{proof}[Proof of Theorem~\ref{thm:r_constant}, item~$(i)$]  By~\eqref{eq:garfre}, it is enough to show that $k_{1/2-q}^{\lambda} = k_{1/2+q}^{\lambda}$ for all $\lambda \in \R$ and $q\in \R$. 
  For the simplicity of notations, we note $D_H:=\scal{\sqrt{D}}_H$.
  Using Propositions~\ref{ppn:equiv_unif_laplacian} and~\ref{prop:left_right}, we have
\begin{equation} \label{eq:k1/2}
	k_{1/2-q}^{\lambda} = k_{1/2-q}^{-\lambda} =  \widehat{k}_{1/2-q}^{-\lambda{D}_H},
\end{equation}
where $\widehat{k}_{1/2-q}^{-\lambda{D}_H}$ is the principal eigenvalue of the operator $\widehat{\mcl}_{1/2-q}^{-\lambda{D}_H}$ defined in Proposition~\ref{ppn:equiv_unif_laplacian}, namely (pointing out that $s_{1/2-q}=q$):
\[\widehat{\mcl}_{1/2-q}^{-\tilde{\lambda}} :\Phi\mapsto\dr_{yy}\Phi - \dr_y\cro{\pth{-2\tilde{\lambda}+ q\drift'}\Phi}+\left(R - \tilde{\lambda} q \drift'+\tilde{\lambda}^2\right)\Phi,\]
with $\tilde{\lambda} := \lambda {D}_H$. 
We apply Proposition~\ref{prop:VP} with $b(x) := -2\tilde{\lambda} + q\drift'$
and
$a(x) := R - \tilde{\lambda} q \drift' + \tilde{\lambda}^2$,
which can be rewritten: \[a(x)=(R-\tilde{\lambda}^2)-\tilde{\lambda}b.\] 
Here, $R-\tilde{\lambda}^2$ and $\tilde{\lambda}$ are constant.
We deduce that $\widehat{k}_{1/2-q}^{-\lambda{D}_H}$ is also the principal eigenvalue of the operator 
\[\Phi\mapsto\dr_{yy}\Phi + \dr_y\cro{\pth{-2\tilde{\lambda}+ q\drift'}\Phi}+\left(R - \tilde{\lambda} q \drift'+\tilde{\lambda}^2\right)\Phi,\]
which is equal to $\widehat{\mcl}_{1/2+q}^{\lambda{D}_H}$. In other words, $\widehat{k}_{1/2-q}^{-\lambda{D}_H} = \widehat{k}_{1/2+q}^{\lambda{D}_H}$. Substituting into \eqref{eq:k1/2} and using Proposition~\ref{ppn:equiv_unif_laplacian} again,  we finally obtain:
\[
k_{1/2-q}^{\lambda} = \widehat{k}_{1/2-q}^{-\lambda{D}_H} = \widehat{k}_{1/2+q}^{\lambda{D}_H} = k_{1/2+q}^{\lambda}.
\]

\end{proof}

\begin{proof}[Proof of Theorem~\ref{thm:r_constant}, items~$(ii)$ and~(iii)]
  We point out that $s_{1/2}=0$, so $\widehat{k}_{1/2}^{\mu}=r+\mu^2$. Therefore,
  \begin{align*}
    c_{1/2}^*&=\inf_{\lambda>0}\frac{{k}_{1/2}^{\lambda}}{\lambda}=\inf_{\lambda>0}\frac{\widehat k_{1/2}^{\lambda\scal{\sqrt{D}}_H}}{\lambda}\\
    &=\inf_{\lambda>0}\pth{\frac{r}{\lambda}+\lambda\pth{\scal{\sqrt{D}}_H}^2}\\
    &=2\sqrt{r}\times{\scal{\sqrt{D}}_H}.
  \end{align*}
  Finally, \cite[Proposition 2.7]{Nad11} implies that $c^*_q\leq c^*_{1/2}$.
  This proves Theorem~\ref{thm:r_constant}, items~$(ii)$ and~$(iii)$.
  
  Let us now prove the remark after the statement of Theorem~\ref{thm:r_constant} about the different periods.
  Consider the solution to the linearized version of~\eqref{eq:main_KPP} in the $L-$periodic setting:
  \[\dr_tv=\dr_{x}(D^{1/2}_L\dr_x(D^{1/2}_Lv))+rv,\]
  where $D_L(x):=D(x/L)$ is the $L-$periodic version of~$D$.
  We denote by $c^*_{1/2}(L)$ the spreading speed of $v$.
  Using the change of variables $x'=x/L$, we obtain that $c^*_{1/2}(L)=Lc'$ where $c'$ is the spreading speed of the solution $u$ of
  \[\dr_tu=\frac{1}{L^2}\dr_{x}(D^{1/2}\dr_x(D^{1/2}u))+ru.\]
  The coefficients of this equation are $1-$periodic.
  Hence, by Theorem~\ref{thm:r_constant}, item~$(ii)$, the spreading speed of $u$ is $c'=2\sqrt{r}\times\scal{\sqrt{D}/L}_H$, where $\scal{\sqrt{D}/L}_H$ is the harmonic mean of $\sqrt{D}/L$.
  Therefore,
  \[c^*_{1/2}(L)=Lc'=2\sqrt{r}\times L\pth{\int_0^1\frac{L}{\sqrt{D(x)}}}^{-1}=2\sqrt{r}\times\scal{\sqrt{D}}_H.\]

\end{proof}

\subsection{Proof of Theorems~\ref{thm:limits} and \ref{thm:monotonicity}}\label{ss:feynman-kac}

Throughout this section, we want to show properties of the principal eigenvalues \allowbreak  $k_q^{\lambda}[r;D]$, with fixed $r$ and $D$. 
Using Proposition~\ref{ppn:equiv_unif_laplacian}, this is equivalent to show properties of the principal eigenvalues $\widehat{k}_q^{\lambda}$. 
Therefore, we will work on the operators $\widehat{\mcl}_q^{\tilde{\lambda}}$, with $\tilde{\lambda}=\lambda\scal{\sqrt{D}}_H$, which will be simpler to handle than $\mcl_q^{\lambda}$.

\begin{proof}[Proof of Theorem~\ref{thm:limits}, item~$(i)$]
  Let us show that the ability of persistence $k_q^0=\widehat{k}_q^0$ converges to the limit stated in the theorem.
  We recall that $\widehat{k}_q^0$ is the principal eigenvalue of the operator defined in Proposition~\ref{ppn:equiv_unif_laplacian}; we recall the notations: $R(y):=r(h^{-1}(y))$ and $\drift(y):=\ln(D(h^{-1}(y))$, where~$h$ is the transformation defined before Proposition~\ref{ppn:equiv_unif_laplacian}.
  We let $\widehat{\mca}:=\acc{y\in\R \,/\, h^{-1}(y)\in\overline{\mca}}$ be the set of local maxima of~$\drift$.
  We want to show that
  \[\lim_{q\to-\infty}\widehat k_q^0=\max_{x\in\overline{\mca}}r(x).\]
  To this aim, we will show that
  \[\lim_{q\to-\infty}\widehat k_q^0=\max_{y\in\widehat{\mca}}R(y).\]
  We consider the adjoint of the problem defining $\widehat k_q^0$, namely, we let $\varphi_q>0$ be a $1-$periodic solution of 
  \begin{equation}\label{eq:limit_adjoint}
    \dr_{yy}\varphi_q+s_q\drift'\dr_y\varphi_q+R{\varphi_q}=\widehat k_q^0\varphi_q.    
  \end{equation}
  We let $\psi_q>0$ be the $1-$periodic function defined by $\psi_q(y):=e^{s_q\drift(y)/2}\varphi_q(y)$.
  By computations similar to those of the proof of Proposition~\ref{prop:compar_k}, we have:
  \[\dr_{yy}\psi_q+\pth{R-\frac{(s_q\drift')^2}{4}-\frac{s_q\drift''}{2}}\psi_q=\widehat k_q^0 \psi_q.\]
  Therefore, by the Rayleigh formula, we have:
  \[\widehat k_q^0=\max_{\phi\in\mce}\pth{-\int_0^1(\phi')^2+\int_0^1\pth{R-\frac{(s_q\drift')^2}{4}-\frac{s_q\drift''}{2}}\phi^2},\]
  where
  \[\mce= \acc{\phi\in H^1_{loc}(\R)\, /\,\dabs{\phi}_{L^2(0,1)}=1,\,\text{$\phi$ is $h(1)-$periodic}}.\]
  This is analogous to Equation~(1.2) in~\cite{CheLou08}. Moreover, since all critical points of $D$ are non-degenerate, we obtain that all critical points of~$\drift$ are non-degenerate.
  Last, $s_q\to+\infty$ as $q\to-\infty$.
  Therefore, the proof of~\cite[Theorem 1]{CheLou08} remains valid in the setting of Equation~\eqref{eq:limit_adjoint}
  (the proof is simpler in our case since there is no need to deal with the boundary condition).
  This yields:
  \[\lim_{q\to-\infty}\widehat k_q^0=\max_{y\in\widehat{\mca}}R(y)=\max_{x\in\overline{\mca}}r(x).\]
The case $q\to +\infty$ is the same by replacing~$\drift'$ with~$-\drift'$.
\end{proof}

Before proving item~$(ii)$ of Theorem~\ref{thm:limits}, let us prove Proposition~\ref{ppn:exit_time}.
\begin{proof}[Proof of Proposition~\ref{ppn:exit_time}]
  We let $b\in\mcc^{0,1}(\R)$ be a function such that $b(x_0)>0$ for some $x_0\in[0,1]$.
  For $s\in\R$, we let $(X_t)_{t\geq 0}$ solve
  \[\de X_t=sb(X_t)\de t+\sqrt{2}\de W_t, \qquad X_0=1.\]
  We let $\mathbb{P}_1$ be the law of $(X_t)_{t\geq0}$.
  Last, we let $\tau:={\inf\acc{t\geq 0\ /\ X_t\leq 0}}$ be the exit time from $(0,+\infty)$ of $(X_t)_{t\geq0}$.
  For clarity, we omit the superindex $s$ in the notations, but the dependency in $s$ should not be forgotten.
  
  As $b(x_0)>0$, there exists an interval $I\subset(0,1)$ such that $b(x) \ge \delta >0$ for all $x\in I$  and some constant $\delta>0$. 
    We define the random variables
  \[\sigma_-:=\inf\acc{u>0 \ /\ X_u\leq\inf I},\qquad \sigma_+:=\sup\acc{u<\sigma_-\ /\ X_u\geq \sup I},\]if the infimum and the supremum are well-defined,  and $\sigma_{\pm}=+\infty$ otherwise.
  Since $X_0=1$, we have, conditionally on $\tau<+\infty$: the difference $\sigma_--\sigma_+$ is finite and is the duration of the first crossing of the interval $I$, without leaving it, by the process $(X_t)_{t\geq0}$; in particular $\sigma_--\sigma_+\leq\tau$.
  Now, we point out that for all $h\geq0$, we have (conditionally on $\tau<+\infty$):
  \begin{align*}
    X_{\sigma_-}
    &=X_{\sigma_--h}+s\int_0^hb(X_{\sigma_--h'})\de h'+\sqrt{2}(W_{\sigma_-}-W_{\sigma_--h}).
  \end{align*}
  Moreover, conditionally on $\tau<+\infty$ and $\sigma_--\sigma_+>h$, we have $X_u\in I$ for all $u\in[\sigma_--h,\sigma_-]$, so, owing to the definition of $I$,
  \begin{align*}
    X_{\sigma_-}&\geq X_{\sigma_--h}+sh\delta+\sqrt{2}(W_{\sigma_-}-W_{\sigma_--h}).
  \end{align*}
  Note that,  since $X_0=1$,
  we have almost surely: $\tau\geq \sigma_-$ and $\sigma_-\geq\sigma_+\geq 0$.
  Hence, for all $a>0$,
  \begin{align*}
    \mathbb{P}_1(\tau\leq a)
    &=\mathbb{P}_1(\tau\leq a,\,\sigma_--\sigma_+\leq a)\\
    &\leq\mathbb{P}_1(\sigma_-\leq a,\,\sigma_--\sigma_+\leq a)\\
    &=\mathbb{P}_1(\sigma_-\leq a,\,\exists h\in[0,\sigma_-],\, X_{\sigma_-}- X_{\sigma_--h}\leq-\abs{I})\\
    &  { \leq } \mathbb{P}_1(\exists\sigma\in[0,a],\,\exists h\in[0,\sigma],\, X_{\sigma}- X_{\sigma-h}\leq-\abs{I}).
  \end{align*}
  Using the above estimate for $X_{\sigma_-}$, this gives:
  \begin{align*}
    \mathbb{P}_1(\tau\leq a)
    &\leq\mathbb{P}_1(\exists\sigma\in[0,a],\,\exists h\in[0,\sigma],\, \sqrt{2}(W_{\sigma}-W_{\sigma-h})\leq -sh\delta-\abs{I}),
  \end{align*}
  which we rewrite:
  \begin{align*}
    \mathbb{P}_1(\tau\leq a)
    &\leq\mathbb{P}_1\pth{\exists\sigma\in[0,a],\,\exists h\in[0,\sigma],\, \frac{\abs{W_{\sigma}-W_{\sigma-h}}}{h^{1/4}}\geq \frac{1}{\sqrt{2}}\pth{sh^{3/4}\delta+h^{-1/4}\abs{I}}}.
  \end{align*}
  By virtue of the regularity of Brownian motion, there exists a random constant~$C$,
  whose law is independent of~$s$,
  such that almost surely,
  \[\sup_{\sigma \in (0,\,a),\ h\in(0,\,\sigma)}\frac{\abs{W_{\sigma}-W_{\sigma-h}}}{h^{1/4}}\leq C.\]
  Hence
  \begin{align*}
    \mathbb{P}_1(\tau\leq a)
    &\leq\mathbb{P}_1\pth{\exists h\in[0,a],\, C\geq\frac{1}{\sqrt{2}}\pth{sh^{3/4}\delta+h^{-1/4}\abs{I}}}.
  \end{align*}
  But, as $s\to+\infty$, we have
  \[\inf_{h\in[0,a]}\pth{sh^{3/4}\delta+h^{-1/4}\abs{I}}\to+\infty.\]
  Thus $\mathbb{P}_1(\tau\leq a)\to 0$ as $s\to+\infty$.

\end{proof}

\begin{proof}[Proof of Theorem~\ref{thm:limits}, item~$(ii)$] 
  Let us show that the spreading speed $c^*_q[r;D]$ converges to~$0$ as $q\to\pm\infty$. 
  Since the KPP assumption~\eqref{f1} holds,
  the spreading speed is linearly determined;
  thus,  $c^*_q[r;D]$ is also the spreading speed of the level sets of the solution~$v$ to the following linear equation:
  \[\dr_tv=\mcl_q^0v=\dr_{xx}(D(x)v)-(1-q)\dr_x(D'v)+rv, \qquad v(0,x)=1_{\acc{x\leq0}}.\]
  Define $h(x):=\int_0^x\de y/\sqrt{D(y)}$ and let
  \[V(t,y):=\sqrt{D(h^{-1}(y))}\,v(t,h^{-1}(y)).\]
  Using Lemma~\ref{lem:space_deformation} and the computations of the proof of Proposition~\ref{ppn:equiv_unif_laplacian}, we have:
  \[\dr_tV=\dr_{yy}V-s_q\dr_y(\drift'V)+RV,\]
  with $s_q:=1/2-q$, $\drift(y):=\ln(D(h^{-1}(y)))$ and $R(y):=r(h^{-1}(y))$.
    Define 
    \[c[V]:=\limsup_{t\to+\infty}\frac{1}{t}\sup\acc{x\in\R \ /\ V(t,x)\geq 1/2}.\]
    Then, since $V(t,y)=\sqrt{D(h^{-1}(y))}\,v(t,h^{-1}(y))$ and $\sqrt{D}$ is bounded, and bounded from below by a positive constant, we have: $c[V]=h(1)\,c^*_q[r;D]$.
    Therefore,
proving that $c^*_q[r;D]$ converges to~$0$ as $q\to\pm\infty$ is equivalent to proving that~$c[V]$ converges to~$0$ as $q\to\pm\infty$.

    Let us now consider the adjoint problem:
    \[\dr_tV^*=\dr_{yy}V^*+s_q\drift'\dr_yV^*+RV^*.\]
    By the standard Freidlin-Gärtner formula, the spreading speed $c[V^*]$ of the level set~$\frac{1}{2}$ of~$V^*$ (defined as~$c[V]$ but with~$V^*$ instead of~$V$) is 
    \[c[V^*]=\inf_{\lambda>0}\frac{\widehat k_q^{\lambda}[V^*]}{\lambda},\]
    where $\widehat k_q^{\lambda}[V^*]$ is the principal eigenvalue of the operator
    \[\widehat\mcl_q^{\lambda}[V^*]\,:\,\Phi\mapsto\dr_{yy}\Phi+(s_q\drift'-2\lambda)\dr_y\Phi+(R-\lambda s_q\drift'+\lambda^2)\Phi.\]
    We point out that $\widehat\mcl_q^{\lambda}[V^*]=(\widehat\mcl_q^{-\lambda})^*$ (where $\widehat\mcl_q^{-\lambda}$ is defined in Proposition~\ref{ppn:equiv_unif_laplacian}),
    so $\widehat k_q^{\lambda}[V^*]=\widehat k_q^{-\lambda}$. 
    Owing to Proposition~\ref{prop:left_right} (which can also be applied with~$\widehat k_q^{\lambda}$ instead of~$k_q^{\lambda}$), we have: 
    $\widehat k_q^{-\lambda}=\widehat k_q^{\lambda}$,
    so $\widehat k_q^{\lambda}[V^*]=k_q^{\lambda}$.
    We conclude that $c[V]=c[V^*]$.

    Hence, we wish to prove that~$c[V^*]$ converges to~$0$ as $q\to\pm\infty$.
    To this aim, we will prove that
    for all $\gamma>0$, for~$\abs{q}$ large enough, we have: $V^*(t,\gamma t)\to 0$ as $t\to+\infty$.
  For the remaining of the proof, we fix $\gamma>0$.
  For conciseness, we note $s=s_q$ and consider $q\to-\infty$, so $s\to+\infty$.

\medskip

  First, we estimate $V^*(t,\gamma t)$.
  By the Feynman-Kac formula~\eqref{eq:feynman_kac_formula}, we have: 
  \begin{equation*}
  V^*(t,y)=\mathbb{E}_y\cro{1_{\acc{X_t\leq 0}}\exp\pth{\int_0^tR(X_{t-s})\de s}},
  \end{equation*}
  where  $\mathbb{E}_y$ is the expectation corresponding to the probability $\mathbb{P}_y$, which is the law of a solution $(X_t)_{t\geq 0}$ to 
  \[\de X_t = s\drift'(X_t)\de t+\sqrt{2}\de W_t,\qquad X_0=y,\]
  for a standard Brownian motion $(W_t)_{t\geq0}$
  (we do not make it explicit in the notations that~$V^*$ and the law of~$X$ depend on~$s$, but this should not be forgotten).
  Letting $r_m:=\max r$, we obtain:
  \begin{equation}\label{eq:upper_bound_fkac}
  V^*(t,y)\leq e^{r_mt}\,\mathbb{P}_y(X_t\leq 0).
  \end{equation}
  We will prove that for all $\gamma>0$, for~$s$ large enough,  $e^{r_mt}\mathbb{P}_{\gamma t}(X_t\leq0)\to0$ as $t\to+\infty$.
  Using~\eqref{eq:upper_bound_fkac}, we will then be able to conclude.

  \medskip
  
  We define, for $i\in\Z$, the random variable
  \[\tau_i:=\inf\acc{u\geq0 \ /\ X_u=i}\] 
   if the infimum is well-defined, and $\tau_i:=+\infty$ otherwise.
  For $0\leq i\leq \lfloor\gamma t\rfloor-1$, when $X_0=\gamma t>1$  and conditionally on $X_t\leq 0$, the difference $\tau_i-\tau_{i+1}$ is positive
  almost surely and corresponds to the
  time spent by the particle~$X_t$ to cross the interval $[i,i+1]$
  for the first time (from right to left).
  We point out that, using the Markov inequality, we have:
  \[\mathbb{P}_{\gamma t}\pth{X_t\leq0}
  { \leq } 
  \mathbb{P}_{\gamma t}\pth{\tau_0\leq t}\leq e^t\mathbb{E}_{\gamma t}[e^{-\tau_0}].
  \]
  Moreover, the $\tau_i-\tau_{i+1}$ are independent and have the same law.
  Hence, for $\gamma t>2$,
  \begin{align*}
    \mathbb{E}_{\gamma t}[e^{-\tau_0}]
    &\leq\mathbb{E}_{\gamma t}\cro{\exp\pth{-\sum_{i=0}^{\lfloor\gamma t\rfloor-1}(\tau_i-\tau_{i+1})}}
    =\pth{\mathbb{E}_{\gamma t}\cro{e^{-(\tau_0-\tau_1)}}}^{\lfloor\gamma t\rfloor}\\
    &\leq\pth{\mathbb{E}_{\gamma t}\cro{e^{-(\tau_0-\tau_1)}}}^{\gamma t/2}\\
    &=\pth{\mathbb{E}_{1}\cro{e^{-(\tau_0-\tau_1)}}}^{\gamma t/2}.
  \end{align*}
  On the last line, contrarily to the previous lines, the expectation is taken for a process starting from~$1$ (as indicated by the index~\enquote{1} on the expectation symbol).
  The above computations show that:
  \begin{equation}\label{eq:compare_exp_moment}
    e^{r_mt}\,\mathbb{P}_{\gamma t}\pth{X_t\leq0} 
  \leq\exp\cro{t\pth{r_m+1+\frac{\gamma}{2}\ln\pth{\mathbb{E}_{1}\cro{e^{-(\tau_0-\tau_1)}}}}}.
  \end{equation}
  Our goal is now to show that, for~$s$ large enough,
  we have $r_m+1+\frac{\gamma}{2}\ln\pth{\mathbb{E}_1\cro{e^{-(\tau_0-\tau_1)}}}<0$.
  We will prove the stronger result that $\mathbb{E}_1\cro{e^{-(\tau_0-\tau_1)}}$ converges to~$0$ as $s\to+\infty$.
  We point out that $\mathbb{P}_1(\tau_1=0)=1$.
  Hence we may write:
  \[\mathbb{E}_1\cro{e^{-(\tau_0-\tau_1)}}
  =\mathbb{E}_1\cro{e^{-\tau_0}}
  =\int_0^1\mathbb{P}_1(e^{-\tau_0}\geq u)\de u
  =\int_0^1\mathbb{P}_1(\tau_0\leq -\ln(u))\de u.\]
  Since $D$ is periodic and nonconstant, $\drift'$ changes sign, so in particular there exists $x_0\in[0,1]$ such that $\drift'(x_0)>0$.
  Therefore, we may apply~Proposition~\ref{ppn:exit_time} with $b=\drift'$. We get: for all $a>0$,
  \begin{equation*}
    \mathbb{P}_1(\tau_0\leq a)\to0 \qquad\text{as}\qquad s\to+\infty.
  \end{equation*}
  Therefore, $\mathbb{E}_1\cro{e^{-(\tau_0-\tau_1)}}\to0$ as $s\to+\infty$.
  Owing to~\eqref{eq:compare_exp_moment}, we obtain that for~$s$ large enough,
  $e^{r_mt}\,\mathbb{P}_{\gamma t}\pth{X_t\leq0}\to 0$ as $t\to+\infty$.
  Owing to~\eqref{eq:upper_bound_fkac}, therefore, we have:
  \[V^*(t,\gamma t)\to 0\qquad\text{as $t\to+\infty$}.\]
  This implies that $c[V]=c[V^*]\leq\gamma$ for $s$ large enough, i.e. for $-q$ large enough.
  Since~$\gamma$ is arbitrary, we get: $c[V]\to 0$ as $q\to-\infty$, and thus $c^*_q[r;D]\to 0$ as $q\to-\infty$.

  \medskip
  
  The same argument holds for $q\to+\infty$, by considering~$-\drift'$ instead of~$\drift'$.
\end{proof}

We now show the monotonicity properties of Theorem~\ref{thm:monotonicity}.

\begin{proof}[Proof of Theorem~\ref{thm:monotonicity}, item~$(i)$]
By definition of~$\kappa(q)$, there exists $\psi>0$ such that
\[\dr_{x}(D\dr_x\psi)+q\dr_x(D'\psi)+(r-\kappa(q))\psi=0.\]
Recall that $r$ is constant. Integrating the equation on $[0,1]$ gives, using the periodicity:
\[(r-\kappa(q))\int_0^1\psi=0.\]
Since $\psi>0$, we have $r=\kappa(q)$.

\end{proof}

\begin{proof}[Proof of Theorem \ref{thm:monotonicity}, item~$(ii)$]
  We let $r$ and $D$ satisfy the assumptions of item~$(ii)$, namely: 
  $D\in\mcc^{3}(\R)$ and $r\in\mcc^4(\R)$ are $1-$periodic and even on $\R$, and monotonic on $[0,1/2]$; 
     $r'\neq 0$ on $(0,1/2)$; $r''(0)\neq 0$ and $r''(1/2)\neq 0$. 

  For convenience, we assume that $r$ and $D$ are nonincreasing on $[0,1/2]$
  (the other cases are proved in an analogous way).
  Therefore, we have:
  \[r'(x)<0\text{ on $(0,1/2)$}, \qquad r'(x)>0\text{ on $(1/2,1)$}, \qquad r'(0)=r'(1/2)=0.\]
  Further, we have $r''(0)<0$ and $r''(1/2)>0$.
  
  Let us prove that $q\mapsto \kappa(q)=k^0_q[r;D]$ is nonincreasing.
  This is equivalent to proving that $q\mapsto\widehat{k}_q^0$ is nonincreasing, where~$\widehat{k}_q^0$ is defined in Proposition~\ref{ppn:equiv_unif_laplacian} as the principal eigenvalue of the operator
  \[\widehat{\mcl}_q^0\,:\,\Phi\mapsto\dr_{yy}\Phi-s_q\dr_y(\drift'\Phi)+R\Phi,\]
  with $\drift(y)=\ln(D(h^{-1}(y))$, $R(y)=r(h^{-1}(y))$ and $s_q=\frac{1}{2}-q$, where~$h$ is defined before Proposition~\ref{ppn:equiv_unif_laplacian}.
  To avoid confusions, we rescale space by a constant in such a way that $h(0)=0$ and $h(1)=1$; since $D$ is even and $1-$periodic, this implies also that $h(1/2)=1/2$. Then~$R$ and~$\drift$ are still even and $1-$periodic.

  We let $q^1>q^2$ and denote $s^1:=s_{q^1}$ and $s^2:=s_{q^2}$, so $s^1<s^2$.
  We define two processes $X^1$ and $X^2$ by:
  \begin{align*}
    \de X_t^1 &= s^1\drift'(X^1_t)\de t+\sqrt{2}\de W^1_t,\esp X_0^1=0,\\
    \de X_t^2 &= s^2\drift'(X^2_t)\de t+\sqrt{2}\de W^2_t,\esp X_0^2=0,
  \end{align*}
  where $(W^1_t)_{t\geq 0}$ and $(W^2_t)_{t\geq 0}$ are two standard Brownian motions. In the following, we will couple the processes $(R(X^i_t))_{t\geq0}$.
First, let us establish a connection between the processes $(R(X^i_t))_{t\geq 0}$ and the principal eigenvalues~$\widehat{k}_{q_i}^0$.

  \paragraph{Step 1. Use the Feynman-Kac formula.}
We now establish a connection between the processes $(R(X^i_t))_{t\geq 0}$ and the principal eigenvalues~$\widehat{k}_{q_i}^0$
by using the Feynman-Kac formula~\eqref{eq:feynman_kac_formula}. 
We work with the adjoint problem.
We point out that the principal eigenvalue of $(\widehat\mcl_q^0)^*$ is equal to the principal eigenvalue of $\widehat\mcl_q^0$, which is precisely $\widehat k_q^0$. Therefore,
for $i=1,2$, we have
\[\cro{(\widehat\mcl_{q^i}^0)^*-\widehat k_{q^i}^0}\varphi^*_i=\dr_{xx}\varphi^*_i+s^i\drift'\dr_x\varphi^*_i+(R-\widehat{k}_{q^i}^0)\varphi^*_i=0,\]
  where $\varphi^*_i$ is the principal eigenfunction of $(\widehat\mcl_{q^i}^0)^*$ such that $\varphi^*_i(0)=1$.
  By the Feynman-Kac formula~\eqref{eq:feynman_kac_formula} (applied to the stationary function $\overline{u}(t,x)=\varphi^*_i(x)$),
  therefore, we have, for all $t\geq0$,
  \[
  \varphi^*_i(x)=\mathbb{E}_x\cro{\varphi^*_i(X^i_t)\exp\pth{\int_0^t(R(X^i_{t-s})-\widehat k_{q^i}^0)\de s}}.
  \]
  Since $0<\varphi^*_i<+\infty$, there is a constant $C>1$ such that, for $i=1,2$, for all $t\geq0$,
  \begin{align*}
    &\frac{1}{C}\exp\pth{-\widehat k_{q^i}^0t}\mathbb{E}_x\cro{\exp\pth{\int_0^tR(X^i_{t-s})\de s}}\\
    &\qquad\leq
    \varphi^*_i(x)
    \leq C\exp\pth{-\widehat k_{q^i}^0t}\mathbb{E}_x\cro{\exp\pth{\int_0^tR(X^i_{t-s})\de s}}.
  \end{align*}
  Therefore, for $i=1,2$,
  \begin{align*}
    0&<
    \liminf_{t\to+\infty}\exp\pth{-\widehat k_{q^i}^0t}\mathbb{E}_x\cro{\exp\pth{\int_0^tR(X^i_{t-s})\de s}}\\
    &\leq \limsup_{t\to+\infty}\exp\pth{-\widehat k_{q^i}^0t}\mathbb{E}_x\cro{\exp\pth{\int_0^tR(X^i_{t-s})\de s}}<+\infty.
  \end{align*}
  Hence, to prove that $\widehat k_{q^1}^0\leq \widehat k_{q^2}^0$, it is sufficient to prove that
  \[\mathbb{E}_x\cro{\exp\pth{\int_0^tR(X^1_{t-s})\de s}}\leq \mathbb{E}_x\cro{\exp\pth{\int_0^tR(X^2_{t-s})\de s}}.\]
  The goal of the following step is to prove this inequality.
  We will construct a  process $(Y^2_t)_{t\geq 0}$ such that:
  $(Y_t^2)_{t\geq0}=(R(X^2_t))_{t\geq 0}$ in law, and such that almost surely, for all $t\geq0$, $R(X_t^1)\leq Y_t^2$.

  \begin{figure}

    \begin{center}
      \includegraphics[width=7cm]{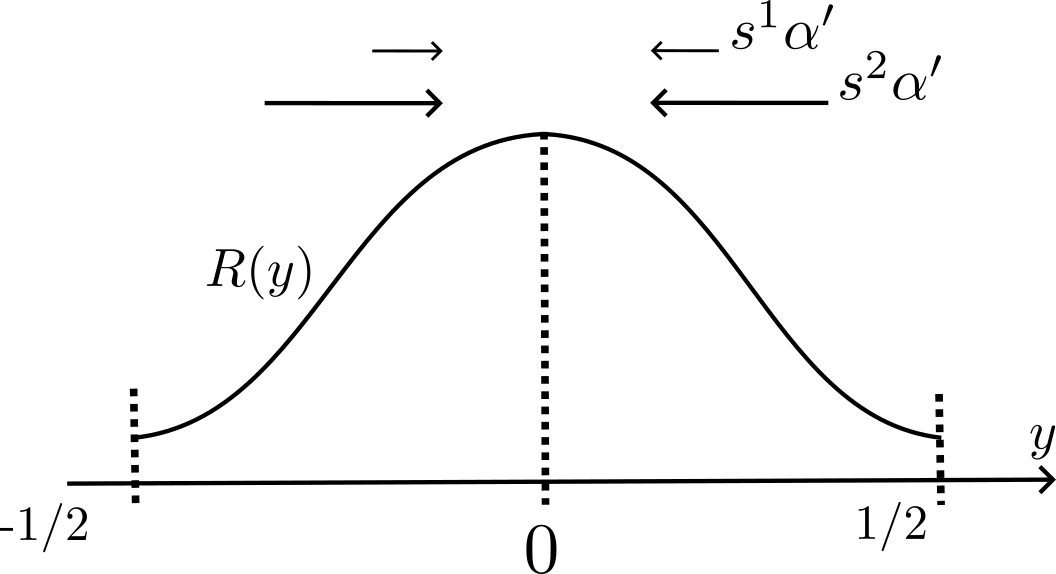}      
    \end{center}
    
    \caption{The processes $X^1$ and $X^2$ are submitted to a drift bringing them to the center, where $R$ is larger. The process $X^2$ is submitted to a stronger drift than $X^1$, which is what we want to exploit in the proof.}
    
  \end{figure}

  \paragraph{Step 2. Write an equation for the process $(R(X^1_t))_{t\geq 0}$.}
  For $x\in\R$, we let $\acc{x}$ be analogous to the fractional part of~$x$,
  but belonging to $(-1/2,1/2]$ instead of $[0,1)$, namely, $\acc{x}$ is the unique real number such that:
    \[\acc{x}\in(-1/2,1/2],\qquad x-\acc{x}\in\Z.\]
  We define the sign function $sgn$ by: $sgn(x)=1$ if $x\geq 0$ and $sgn(x)=-1$ if $x<0$.
  By our symmetry assumptions on $r$ and $D$, we have $r(x)=r(\abs{\acc{x}})$, so
  \[R(x)=R(\abs{\acc{x}}), \qquad R'(x)=sgn(\acc{x})R'(\abs{\acc{x}}),\]
  and
  $D(x)=D(\abs{\acc{x}})$, 
  so
  \[\drift'(x)=sgn(\acc{x})\drift'(\abs{\acc{x}}).\]
  Now, we set $Y^1_t:=R(X^1_t)$. Applying Itô's formula, we have:
  \begin{align*}
    \de Y^1_t=\cro{s^1R'(X^1_t)\drift'(X^1_t)+R''(X^1_t)}\de t+\sqrt{2}\,R'(X^1_t)\de W^1_t.
  \end{align*}
  By our above remarks, we get:
  \begin{align*}
    \de Y^1_t=\cro{s^1R'(\abs{\acc{X^1_t}})\drift'(\abs{\acc{X^1_t}})+R''(\abs{\acc{X^1_t}})}\de t+\sqrt{2}\,sgn(\acc{X^1_t})R'(\abs{\acc{X^1_t}})\de W^1_t.
  \end{align*}
  By our monotonicity assumptions on $R$, the function $x\mapsto R(x)$ is a bijection from $[0,1/2]$ to $[R(1/2),R(0)]$.
  Therefore, since $x\mapsto\abs{\acc{x}}$ takes its values in $[0,1/2]$, we can write:
  \[\abs{\acc{X^1_t}}=R^{-1}(Y^1_t).\]
  We obtain that $Y^1_t$ satisfies
  \begin{equation}\label{eq:sde_y1}
    \de Y^1_t=\cro{s^1b(Y^1_t)+B(Y^1_t)}\de t+\sigma(Y^1_t)\de \overline{W}_t,
  \end{equation}
  where
  \[\overline{W}_t:=\int_0^tsgn(\acc{X^1_t})\de W^1_t\]
  and, for $y\in[R(1/2),R(0)]$,
  \[b(y):=R'(R^{-1}(y))\drift'(R^{-1}(y)),\qquad B(y):=R''(R^{-1}(y)),\qquad \sigma(y):=\sqrt{2}\,R'(R^{-1}(y)).\]
  Last, $(\overline{W}_t)_{t\geq0}$ is a martingale with increasing process
  \[\scal{\overline{W}}_t=\int_0^tsgn(\acc{X^1_t})^2\de t=t,\]
  so $(\overline{W}_t)_{t\geq0}$ is a Brownian motion.

  \paragraph{Step 3. Regularity of the coefficients of~\eqref{eq:sde_y1}.}
  We prove that $b$, $B$ and $\sigma^2$ are Lipschitz.
  We will see in Step~4 that these conditions ensure the existence and uniqueness of a solution to~\eqref{eq:sde_y1}.

  First, $R(y)=r(h^{-1}(y))$, so $R\in\mcc^4(\R)$, and
  \[R'(y)=\frac{r'(h^{-1}(y))}{h'(h^{-1}(y))}=\sqrt{D(h^{-1}(y))}\,r'(h^{-1}(y))\]
  is nonzero on $(0,1/2)$.
  Moreover,
  \[R''(y)=\frac{1}{h'(h^{-1}(y))}\pth{\frac{D'(h^{-1}(y))\,r'(h^{-1}(y))}{2\sqrt{D(h^{-1}(y))}}+\sqrt{D(h^{-1}(y))}\,r''(h^{-1}(y))}.\]
  Since $D'(0)=D'(1/2)=0$, we have, at $y=0$ and at $y=1/2$:
  \[R''(y)=\frac{1}{h'(h^{-1}(y))}\,\sqrt{D(h^{-1}(y))}\,r''(h^{-1}(y)),\]
  which, by our assumptions on~$r$, is nonzero.
  Further, $R$ shares the same monotonicity properties as~$r$.
  Hence, properties $(a)$, $(b)$ and $(c)$ of the statement also hold with~$R$ instead of~$r$.
  
  Second, we have:
  \[b'(y)=\frac{R''(R^{-1}(y))}{R'(R^{-1}(y))}\drift'(R^{-1}(y))+\drift''(R^{-1}(y)).\]
  Since~$R$ and~$\drift$ are of class $\mcc^2$, $R''$ and~$\drift''$ are bounded.
  Further, $R'(0)=R'(1/2)=0$ and $R'$ is nonzero on $(0,1/2)$.
  Since $R''(0)\neq 0$ and $\drift'(0)=0$, we have, for $x\to 0$:
  \[\abs{\frac{R''(x)}{R'(x)}\drift'(x)}\leq\frac{R''(0)}{R''(0)x+O(x^2)}\ O(x),\]
  which is bounded.
  Hence $b'$ is bounded near $0$.
  Likewise, $b'$ is bounded near $1/2$. Therefore, $b'$ is bounded on $[0,1/2]$ so $b$ is Lipschitz.

  Third, we have
  \[B'(y)=\frac{R'''(R^{-1}(y))}{R'(R^{-1}(y))}.\]
  The function $R'''$ is bounded on $[0,1/2]$ and $R'$ is nonzero on $(0,1/2)$.
  Moreover, owing to $R\in\mcc^4(\R)$, we have, as $x\to0$, $R'''(x)=O(x)$ and $R'(x)=R''(0)x+O(x^2)$.
  Since $R''(0)\neq0$, $R'''/R'$ is bounded near~$0$.
  Likewise, $R'''/R'$ is bounded near~$1/2$.
  This implies that $B'$ is bounded on  $[0,1/2]$, so $B$ is Lipschitz.

  Last, we have:
  \begin{align*}
  (\sigma^2)'(y)&=2\sigma(y)\sigma'(y)=\frac{4R'(R^{-1}(y))R''(R^{-1}(y))}{R'(R^{-1}(y))}=4R''(R^{-1}(y)),  
  \end{align*}
  which is bounded. Therefore, $\sigma^2$ is Lipschitz.
  
  \paragraph{Step 4. Conclusion.}
  By Step 3, $s^1b+B$ and $\sigma^2$ are Lipschitz on $[R(1/2),R(0)]$.
  We extend the functions $b$, $B$ and $\sigma^2$ to Lipschitz functions on $\R$, with $b\geq 0$.
  Then, using~\cite[Theorem~IX.3.5.$(ii)$]{RevYor99} (with $\rho(x):=\abs{x}$), we conclude that pathwise uniqueness holds for Equation~\eqref{eq:sde_y1}.
  
  Now, let us consider the following equation:
  \begin{equation}\label{eq:sde_y2}
    \de Y^2_t=\cro{s^2b(Y^2_t)+B(Y^2_t)}\de t+\sigma(Y^2_t)\de W_t.
  \end{equation}
  As for~\eqref{eq:sde_y1}, pathwise uniqueness holds for~\eqref{eq:sde_y2}.
  Further, doing the same computations as in Step~2, there exists a Brownian motion $(\overline{W}')_{t\geq0}$ (constructed as $(\overline{W}_t)_{t\geq0}$ with $(W^2)_{t\geq0}$ instead of $(W^1)_{t\geq0}$) such that $(R(X^2_t))_{t\geq 0}$ solves~\eqref{eq:sde_y2} with $W=\overline{W}'$.
  
  Hence, Equation~\eqref{eq:sde_y2} has a solution; owing to pathwise uniqueness,
  it follows from~\cite{RevYor99} (see Theorem IX.1.7 and the associated Remark~2)
  that this solution is strong, and, therefore, that there exists a (unique) solution $(Y^2_t)_{t\geq0}$ to~\eqref{eq:sde_y2} carried by the Brownian motion $W=\overline{W}$ defined in Step~2.
  By uniqueness in law, we have:
  \[(Y^2_{t})_{t\geq0}=(R(X^2_{t}))_{t\geq 0}\qquad\text{in law.}\]
  We point out that the processes $(Y^1_t)_{t\geq0}$ and $(Y^2_t)_{t\geq0}$ are carried by the same Brownian motion, which allows us to compare them.
  We have indeed, on $\R$,
  \[s^1b+B\leq s^2b+B.\]
  Hence, the comparison result stated in~\cite[Theorem~IX.3.7]{RevYor99}, implies that almost surely,
  \[\forall t\geq 0,\qquad Y^1_t\leq Y^2_t.\]
  Therefore,
  \begin{equation*}
    \mathbb{E}\cro{\exp\pth{\int_0^tY^1_t\de y}}\leq\mathbb{E}\cro{\exp\pth{\int_0^tY^2_t\de y}}.
  \end{equation*}
Recalling Step 1, this concludes the proof.

       \end{proof}

\begin{proof}[Proof of Theorem \ref{thm:monotonicity}, item~$(iii)$]
  Let $D$ be nonconstant and, on $[0,1)$, have exactly one local minimum at $\underline{x}\in[0,1)$ and one local maximum at $\overline{x}\in[0,1)$.
        We let $\eps>0$ and define
        \[r_{\eps}(x):=
\left\{
\begin{aligned}
  &1&\esp&x\in(\underline{x}-\eps,\underline{x}+\eps)\cup(\overline{x}-\eps,\overline{x}+\eps),\\
  &0&\esp&\text{otherwise}.
\end{aligned}
\right.
\]
By Theorem \ref{thm:limits}, we have for all $\eps>0$,
\[\lim_{q\to-\infty}k_q^0[r_{\eps},D]=\lim_{q\to+\infty}k_q^0[r_{\eps},D]=1.\]
Thus, for all $\eps>0$, $q\mapsto k_q^0[r_{\eps},D]$ cannot be monotonic unless it is constant.
But $r_{\eps}\to 0$ in $L^2([0,1])$ as $\eps\to0$.
Therefore, for all $q\in\R$, $k_q^0[r_{\eps},D]\to 0$ as $\eps\to 0$.
Hence, for $\eps>0$ small enough, $q\mapsto k_q^0[r_{\eps},D]$ is not monotonic.

\end{proof}

\begin{proof}[Proof of Theorem \ref{thm:monotonicity}, item~$(iv)$]
  By Theorem \ref{thm:limits}, we have:
  \[\lim_{q\to-\infty}c_q^*[r;D]=\lim_{q\to+\infty}c_q^*[r;D]=0.\]
  On the other hand, for $q\in\R$ such that $k^0_q[r;D]>0$, we have $c_q^*[r;D]>0$. This proves that $q\mapsto c_q^*[r;D]$ cannot be monotonic.
\end{proof}

\section*{Acknowledgments} 
N.B. and L.R. were supported by the ANR project ReaCh, {ANR-23-CE40-0023-01}.
N.B. was supported by the Chaire Modélisation Mathématique et Biodiversité (École Polytechnique, Muséum national d’Histoire naturelle, Fondation de l’École Polytechnique, VEOLIA Environnement). Y.-J.K. was supported by the National Research Foundation of Korea (RS-2024-00347311). In-person collaborations between N.B., Y.-J.K., and L.R. were partially funded by the International Research Network ReaDiNet and the INRAE Network MEDIA.

\section*{Data Availability}  
The code used to perform the numerical simulations presented in this manuscript is publicly available at \href{https://doi.org/10.17605/OSF.IO/GDQVP}{DOI: 10.17605/OSF.IO/GDQVP}.

\bibliographystyle{abbrv}

\begin{thebibliography}{10}

\bibitem{AlfGil22}
M.~Alfaro, T.~Giletti, Y.-J. Kim, G.~Peltier, and H.~Seo.
\newblock On the modelling of spatially heterogeneous nonlocal diffusion:
  deciding factors and preferential position of individuals.
\newblock {\em J. Math. Biol.}, 84(5):Paper No. 38, 35, 2022.

\bibitem{BenFeu14}
M.~Bengfort, U.~Feudel, F.~M. Hilker, and H.~Malchow.
\newblock Plankton blooms and patchiness generated by heterogeneous physical
  environments.
\newblock {\em Ecological complexity}, 20:185--194, 2014.

\bibitem{BenMal16}
M.~Bengfort, H.~Malchow, and F.~M. Hilker.
\newblock The {F}okker-{P}lanck law of diffusion and pattern formation in
  heterogeneous environments.
\newblock {\em J. Math. Biol.}, 73(3):683--704, 2016.

\bibitem{Berhamnad08}
H.~Berestycki, F.~Hamel, and G.~Nadin.
\newblock Asymptotic spreading in heterogeneous diffusive excitable media.
\newblock {\em J. Funct. Anal.}, 255(9):2146--2189, 2008.

\bibitem{BerHamNad05d}
H.~Berestycki, F.~Hamel, and N.~Nadirashvili.
\newblock The speed of propagation for {KPP} type problems. {I}. {P}eriodic
  framework.
\newblock {\em J. Eur. Math. Soc. (JEMS)}, 7(2):173--213, 2005.

\bibitem{BerHamRoq05a}
H.~Berestycki, F.~Hamel, and L.~Roques.
\newblock Analysis of the periodically fragmented environment model. {I}.
  {S}pecies persistence.
\newblock {\em J. Math. Biol.}, 51(1):75--113, 2005.

\bibitem{BerHamRoq05b}
H.~Berestycki, F.~Hamel, and L.~Roques.
\newblock Analysis of the periodically fragmented environment model. {II}.
  {B}iological invasions and pulsating travelling fronts.
\newblock {\em J. Math. Pures Appl. (9)}, 84(8):1101--1146, 2005.

\bibitem{BHR07}
H.~Berestycki, F.~Hamel, and L.~Rossi.
\newblock Liouville-type results for semilinear elliptic equations in unbounded
  domains.
\newblock {\em Ann. Mat. Pura Appl. (4)}, 186(3):469--507, 2007.

\bibitem{BouHamRoq25}
N.~Boutillon, F.~Hamel, and L.~Roques.
\newblock Periodic {KPP} equations: new insights into persistence, spreading,
  and the role of advection.
\newblock {\em arXiv:2503.20390}, 2025.

\bibitem{CanCos03}
R.~S. Cantrell and C.~Cosner.
\newblock {\em Spatial ecology via reaction-diffusion equations}.
\newblock Wiley Series in Mathematical and Computational Biology. John Wiley \&
  Sons, Ltd., Chichester, 2003.

\bibitem{CheLou08}
X.~Chen and Y.~Lou.
\newblock Principal eigenvalue and eigenfunctions of an elliptic operator with
  large advection and its application to a competition model.
\newblock {\em Indiana Univ. Math. J.}, 57(2):627--658, 2008.

\bibitem{Cus09}
E.~L. Cussler.
\newblock {\em Diffusion: mass transfer in fluid systems}.
\newblock Cambridge university press, 2009.

\bibitem{Ein05}
A.~Einstein.
\newblock Über die von der molekularkinetischen theorie der wärme geforderte
  bewegung von in ruhenden flüssigkeiten suspendierten teilchen.
\newblock {\em Annalen der Physik}, 322(8):549--560, 1905.

\bibitem{Far93}
S.~J. Farlow.
\newblock {\em Partial differential equations for scientists and engineers}.
\newblock Courier Corporation, 1993.

\bibitem{Fic55}
A.~Fick.
\newblock On liquid diffusion.
\newblock {\em The London, Edinburgh, and Dublin Philosophical Magazine and
  Journal of Science}, 10(63):30--39, 1855.

\bibitem{Gar09}
C.~Gardiner.
\newblock {\em Stochastic methods}.
\newblock Springer Series in Synergetics. Springer-Verlag, Berlin, fourth
  edition, 2009.
\newblock A handbook for the natural and social sciences.

\bibitem{GarFre79}
J.~Gärtner and M.~I. Fre\u{\i}dlin.
\newblock The propagation of concentration waves in periodic and random media.
\newblock {\em Dokl. Akad. Nauk SSSR}, 249(3):521--525, 1979.

\bibitem{HamNadRoq11}
F.~Hamel, G.~Nadin, and L.~Roques.
\newblock A viscosity solution method for the spreading speed formula in slowly
  varying media.
\newblock {\em Indiana Univ. Math. J.}, 60(4):1229--1247, 2011.

\bibitem{HanEkb01}
S.~Hannunen and B.~Ekbom.
\newblock {Host plant influence on movement patterns and subsequent
  distribution of the polyphagous herbivore Lygus rugulipennis (Heteroptera:
  Miridae)}.
\newblock {\em Environmental entomology}, 30(3):517--523, 2001.

\bibitem{HilKan24}
D.~Hilhorst, S.-M. Kang, H.-Y. Kim, and Y.-J. Kim.
\newblock Fick's law selects the {N}eumann boundary condition.
\newblock {\em Nonlinear Anal.}, 245:Paper No. 113561, 15, 2024.

\bibitem{KimLee24}
H.~Kim, K.~K. Lee, F.~Gadisa, J.~Lee, M.~C. Choi, and Y.-J. Kim.
\newblock Fractionation by spatially heterogeneous diffusion: Experiments and
  the two-component random walk model.
\newblock {\em J. Am. Chem. Soc.}, 146(37):25544--25551, 2024.

\bibitem{KimSeo21}
Y.-J. Kim and H.~Seo.
\newblock Model for heterogeneous diffusion.
\newblock {\em SIAM J. Appl. Math.}, 81(2):335--354, 2021.

\bibitem{KinKaw06}
N.~Kinezaki, K.~Kawasaki, and N.~Shigesada.
\newblock Spatial dynamics of invasion in sinusoidally varying environments.
\newblock {\em {Popul Ecol}}, 48:263--270, 2006.

\bibitem{Nad09a}
G.~Nadin.
\newblock The principal eigenvalue of a space-time periodic parabolic operator.
\newblock {\em Ann. Mat. Pura Appl. (4)}, 188(2):269--295, 2009.

\bibitem{Nad10}
G.~Nadin.
\newblock The effect of the {S}chwarz rearrangement on the periodic principal
  eigenvalue of a nonsymmetric operator.
\newblock {\em SIAM J. Math. Anal.}, 41(6):2388--2406, 2010.

\bibitem{Nad11}
G.~Nadin.
\newblock Some dependence results between the spreading speed and the
  coefficients of the space-time periodic {F}isher-{KPP} equation.
\newblock {\em European J. Appl. Math.}, 22(2):169--185, 2011.

\bibitem{Nad15}
G.~Nadin.
\newblock How does the spreading speed associated with the {F}isher-{KPP}
  equation depend on random stationary diffusion and reaction terms?
\newblock {\em Discrete Contin. Dyn. Syst. Ser. B}, 20(6):1785--1803, 2015.

\bibitem{NinNak24}
H.~Ninomiya and K.~Nakajima.
\newblock Propagation and blocking of bistable waves by variable diffusion.
\newblock 2024.
\newblock preprint (Version 1) available at Research Square.

\bibitem{Oku86}
A.~Okubo.
\newblock Dynamical aspects of animal grouping: swarms, schools, flocks, and
  herds.
\newblock {\em Advances in biophysics}, 22:1--94, 1986.

\bibitem{Okulev02}
A.~Okubo and S.~A. Levin.
\newblock {\em Diffusion and Ecological Problems -- Modern Perspectives}.
\newblock {Second edition, Springer-Verlag, New York}, {2002}.

\bibitem{Pat53}
C.~S. Patlak.
\newblock Random walk with persistence and external bias.
\newblock {\em Bull. Math. Biophys.}, 15:311--338, 1953.

\bibitem{PotSch14}
A.~Potapov, U.~E. Schl\"{a}gel, and M.~A. Lewis.
\newblock Evolutionarily stable diffusive dispersal.
\newblock {\em Discrete Contin. Dyn. Syst. Ser. B}, 19(10):3319--3340, 2014.

\bibitem{RevYor99}
D.~Revuz and M.~Yor.
\newblock {\em Continuous Martingales and Brownian Motion}, volume 293 of {\em
  Grundlehren der mathematischen Wissenschaften}.
\newblock Springer, 1999.

\bibitem{Roq13}
L.~Roques.
\newblock {\em Mod{\`e}les de r{\'e}action-diffusion pour l'{\'e}cologie
  spatiale}.
\newblock Editions Quae, 2013.

\bibitem{RoqAug08}
L.~Roques, M.-A. Auger-Rozenberg, and A.~Roques.
\newblock Modelling the impact of an invasive insect via reaction-diffusion.
\newblock {\em Math. Biosci.}, 216(1):47--55, 2008.

\bibitem{Shikaw97}
N.~Shigesada and K.~Kawasaki.
\newblock {\em {Biological Invasions: Theory and Practice}}.
\newblock {Oxford Series in Ecology and Evolution, Oxford: Oxford University
  Press}, {1997}.

\bibitem{Tur98}
P.~Turchin.
\newblock {\em Quantitative Analysis of Movement: Measuring and Modeling
  Population Redistribution in Animals and Plants}.
\newblock {Sinauer, Sunderland, MA}, {1998}.

\bibitem{Vis08}
A.~W. Visser.
\newblock {Lagrangian modelling of plankton motion: From deceptively simple
  random walks to Fokker--Planck and back again}.
\newblock {\em Journal of Marine Systems}, 70(3-4):287--299, 2008.

\bibitem{Wei02}
H.~F. Weinberger.
\newblock On spreading speeds and traveling waves for growth and migration
  models in a periodic habitat.
\newblock {\em J. Math. Biol.}, 45(6):511--548, 2002.

\bibitem{Wer14}
M.~T. Wereide.
\newblock La diffusion d'une solution dont la concentration et la
  temp{\'e}rature sont variables.
\newblock In {\em Annales de Physique}, volume~9, pages 67--83. EDP Sciences,
  1914.

\end{thebibliography}

\end{document}